\documentclass[letterpaper,11pt]{amsart}

\usepackage[margin=1.2in]{geometry}
\usepackage{amsmath,amsthm,amssymb}
\usepackage{xspace,xcolor}
\usepackage[breaklinks,colorlinks,citecolor=teal,linkcolor=teal,urlcolor=teal,pagebackref,hyperindex]{hyperref}
\usepackage[alphabetic]{amsrefs}
\usepackage[all]{xy}
\usepackage{color}
\usepackage{tikz-cd}

\theoremstyle{plain}
\newtheorem{thm}{Theorem}[section]

\newtheorem{lem}[thm]{Lemma}
\newtheorem{prop}[thm]{Proposition}
\newtheorem{cor}[thm]{Corollary}

\theoremstyle{definition}
\newtheorem{defi}[thm]{Definition}

\newtheorem{eg}[thm]{Example}

\newtheorem{question}[thm]{Question}

\theoremstyle{remark}
\newtheorem{rmk}[thm]{Remark}

\newtheorem{claim}[thm]{Claim}

\def\Z{{\mathbf Z}}
\def\Q{{\mathbf Q}}

\def\C{{\mathbf C}}

\def\A{{\mathbf A}}
\def\P{{\mathbf P}}

\def\cD{\mathcal{D}}

\def\cH{\mathcal{H}}
\def\cI{\mathcal{I}}
\def\cJ{\mathcal{J}}

\def\cO{\mathcal{O}}

\def\cR{\mathcal{R}}

\def\fra{\mathfrak{a}}

\def\frm{\mathfrak{m}}

\def\bff{{\bf f}}

\def\.{\cdot}
\def\^{\widehat}

\def\de{\partial}

\def\({\left(}
\def\){\right)}

\renewcommand{\and}{ \ \ \text{ and } \ \ }

\newcommand{\factor}[2]{\left. \raise 2pt\hbox{$#1$} \right/\hskip -2pt\raise -2pt\hbox{$#2$}}

\DeclareMathOperator{\lct} {lct}

\makeatletter
\@namedef{subjclassname@2020}{\textup{2020} Mathematics Subject Classification}
\makeatother

\usepackage{soul}

\begin{document}

\author[Q.~Chen]{Qianyu Chen}

\address{Department of Mathematics, University of Michigan, 530 Church Street, Ann Arbor, MI 48109, USA}

\email{qyc@umich.edu}

\author[B.~Dirks]{Bradley Dirks}

\address{Department of Mathematics, Stony Brook University, Stony Brook, NY 11794-3651, USA}

\email{bradley.dirks@stonybrook.edu}

\author[M.~Musta\c{t}\u{a}]{Mircea Musta\c{t}\u{a}}

\address{Department of Mathematics, University of Michigan, 530 Church Street, Ann Arbor, MI 48109, USA}

\email{mmustata@umich.edu}

\author[S.~Olano]{Sebasti\'{a}n Olano}

\address{Department of Mathematics, University of Toronto, 40 St. George St., Toronto, Ontario,
CANADA , M5S 2E4}

\email{seolano@math.toronto.edu}

\thanks{Q.C. was partially supported by NSF grant DMS-1952399, B.D. was partly supported by NSF grant DMS-2001132 and NSF-MSPRF grant DMS-2303070, 
 and M.M. was partially supported by NSF grants DMS-2301463 and DMS-1952399.}


\begin{abstract}
We define and study a notion of \emph{minimal exponent} for a local complete intersection subscheme $Z$ of a smooth complex algebraic variety $X$, extending the invariant defined by Saito in the case of hypersurfaces. Our definition is in terms of the Kashiwara-Malgrange $V$-filtration associated to $Z$. We show that the minimal exponent describes how far the Hodge filtration and order filtration agree on the local cohomology $\cH^r_Z(\cO_X)$, where $r$ is the codimension of $Z$ in $X$. We also study its relation to the Bernstein-Sato polynomial of $Z$. Our main result describes the minimal exponent of a higher codimension subscheme in terms of the invariant associated to a suitable hypersurface; this allows proving the main properties of this invariant by reduction to the codimension $1$ case. A key ingredient for our main result is a description of the Kashiwara-Malgrange $V$-filtration associated to \emph{any} ideal $(f_1,\ldots,f_r)$ in terms of the microlocal $V$-filtration associated to the hypersurface defined by $\sum_{i=1}^rf_iy_i$. 
\end{abstract}

\title[$V$-filtrations and minimal exponents]{$V$-filtrations and minimal exponents for local complete intersections}

\maketitle

\section{Introduction}

Let $X$ be a smooth, irreducible, complex algebraic variety. If $Z$ is a nonempty hypersurface in $X$, then the \emph{minimal exponent} $\widetilde{\alpha}(Z)$ of $Z$
(written also as $\widetilde{\alpha}(f)$ if $Z$ is defined by $f\in\cO_X(X)$)
is an important invariant of the singularities of $Z$ introduced by Saito \cite{Saito_microlocal}. When $Z$ has isolated singularities, it can be described via 
asymptotic expansions of integrals along vanishing cycles and it was studied extensively in the 80s, see for example \cite{Varchenko}, \cite{Steenbrink},
and \cite{Loeser}; in this setting, it has been known as \emph{complex singularity index} or
\emph{Arnold exponent} of $Z$. In general, it is defined as the negative of the largest root of the \emph{reduced Bernstein-Sato polynomial} of $Z$
(with the convention that it is $\infty$ if this polynomial is $1$, which is the case if and only if $Z$ is smooth). By results of Koll\'{a}r \cite{Kollar} and Lichtin
\cite{Lichtin}, it is known that the minimal exponent refines an important invariant of singularities in birational geometry, the \emph{log canonical threshold} ${\rm lct}(X,Z)$;
more precisely, we always have ${\rm lct}(X, Z)=\min\{\widetilde{\alpha}(Z),1\}$. 
Our main goal in this paper is to introduce and study a generalization of the minimal exponent to the
case when $Z$ is locally a complete intersection in $X$.

Before giving the definition in the general context, we recall the connection between the minimal exponent of hypersurfaces and two important $\cD$-module theoretic 
constructions associated to $Z$, the Hodge filtration on the local cohomology $\cH_Z^1(\cO_X)$ of $\cO_X$ along $Z$ and the Kashiwara-Malgrange $V$-filtration associated to $Z$.
Recall that if $Z$ is any closed subscheme of $X$, the local cohomology sheaves $\cH^q_Z(\cO_X)$ underlie mixed Hodge modules in the sense of Saito's theory \cite{Saito_MHM}. 
In particular, they carry  a Hodge filtration: this is an increasing filtration by coherent $\cO_X$-modules which is compatible with the order filtration on the sheaf
$\cD_X$ of differential operators on $X$. If $Z$ is a reduced hypersurface in $X$, then the only nonzero local cohomology is $\cH^1_Z(\cO_X)=\cO_X(*Z)/\cO_X$
(where $\cO_X(*Z)$ is the sheaf of rational functions with poles along $Z$). In this case it is known that for every $k\geq 0$ we have
$$F_k\cH^1_Z(\cO_X)\subseteq O_k\cH_Z^1(\cO_X):=\cO_X\big((k+1)Z\big)/\cO_X$$
and Saito showed in \cite{Saito-MLCT} that 
\begin{equation}\label{equation1_intro}
F_k\cH^1_Z(\cO_X)=O_k\cH^1(\cO_X)\,\,\text{for all}\,\,k\leq p\quad\text{if and only if}\quad\widetilde{\alpha}(Z)\geq p+1.
\end{equation}
Using a refinement of this result to a setting involving twists by rational multiples of $Z$, as well as properties of Hodge filtrations, 
it was shown in \cite{MP1} that one can extend the known properties of the Arnold exponent to arbitrary hypersurface singularities. 

The proof of (\ref{equation1_intro}) makes use of results about $V$-filtrations. Let us briefly recall this notion, due to 
Malgrange \cite{Malgrange} and Kashiwara \cite{Kashiwara}, in the more general context that is relevant to this paper. 
 Working locally, let us suppose that
$Z$ is a closed subscheme of $X$ defined by the ideal generated by nonzero regular functions $f_1,\ldots,f_d\in\cO_X(X)$. 
If $\iota\colon X\hookrightarrow X\times\A^d$ is the graph embedding associated to ${\mathbf f}=(f_1,\ldots,f_d)$, that is,
 $\iota(x)=\big(x,f_1(x),\ldots,f_d(x)\big)$, then the $V$-filtration is a decreasing filtration $(V^{\gamma}B_{\mathbf f})_{\gamma\in\Q}$ on 
 $$B_{\mathbf f}:=\iota_+\cO_X=\bigoplus_{\beta\in\Z_{\geq 0}^d}\cO_X\partial_t^{\beta}\delta_{\mathbf f},$$
 indexed by rational numbers, and characterized by a few properties (for details, see Section~\ref{section_review}). 
In the case of one function, the $V$-filtration plays an important role in the theory of mixed Hodge modules. We also recall that
in the case when we only have one function $g\in\cO_X(X)$, Saito introduced in \cite{Saito_microlocal} a related filtration, 
the \emph{microlocal $V$-filtration}
$(V^{\gamma}\widetilde{B}_g)_{\gamma\in\Q}$ on 
$$\widetilde{B}_g=\bigoplus_{j\in\Z}\cO_X\partial_t^j\delta_g.$$
If $Z$ is the hypersurface defined by $g$, then the minimal exponent $\widetilde{\alpha}(Z)$ is described as follows:
$$\widetilde{\alpha}(Z)=\sup\{\gamma>0\mid \delta_g\in V^{\gamma}\widetilde{B}_g\}$$
(see \cite[(1.3.4)]{Saito-MLCT}). In terms of the usual $V$-filtration, this says that if $q$ is a nonnegative integer and
$\gamma\in (0,1]$ is a rational number, then 
\begin{equation}\label{eq2_intro}
\widetilde{\alpha}(Z)\geq q+\gamma\quad\text{if and only if}\quad \partial_t^q\delta_g\in V^{\gamma}B_g.
\end{equation}

Suppose now that $Z$ is a closed subscheme of $X$ that is a local complete intersection, of pure codimension $r\geq 1$. We define the minimal exponent
$\widetilde{\alpha}(Z)$ such that the analogue of formula of (\ref{eq2_intro}) holds in this setting. Working locally, we may assume that $Z$ is defined
by the ideal generated by $f_1,\ldots,f_r\in\cO_X(X)$. In this case, we put 
\begin{equation}\label{eq3_intro}
\widetilde{\alpha}(Z)
= \left\{
\begin{array}{cl}
\sup\{\gamma>0\mid \delta_{\mathbf f}\in V^{\gamma}B_{\bff}\} , & \text{if}\,\,\delta_{\mathbf f}\not\in V^rB_{\mathbf f}; \\[2mm]
\sup\{r-1+q+\gamma\mid \partial_t^{\beta}\delta_{\mathbf f}\in V^{r-1+\gamma}B_{\bff}\,\,\text{for}\,\,|\beta|\leq q\} , & \text{if}\,\,\delta_{\mathbf f}\in V^rB_{\mathbf f},
\end{array}\right.
\end{equation}
where in the latter case, the supremum is over all nonnegative integers $q$ and all rational numbers $\gamma\in (0,1]$ with the property that
$\partial_t^{\beta}\delta_{\mathbf f}\in V^{r-1+\gamma}B_{\bff}$ for all $\beta=(\beta_1,\ldots,\beta_r)\in\Z_{\geq 0}^r$, with $\beta_1+\ldots+\beta_r\leq q$. 
In fact, the supremum in the definition is a maximum unless $\widetilde{\alpha}(Z)=\infty$ (which we show is the case if and only if $Z$ is smooth). 
We note that by \cite[Theorem~1]{BMS}, which describes the multiplier ideals of $Z$ in terms of $V^{\bullet}B_{\mathbf f}$, we have
$$\lct(X,Z)=\min\{\widetilde{\alpha}(Z),r\}.$$

We note that $\widetilde{\alpha}(Z)$ does depend on the ambient variety and not just on $Z$. Whenever $X$ is not understood from the context, we write
$\widetilde{\alpha}(X,Z)$ in order to avoid confusion. However, the dependence is easy to understand: the difference $\widetilde{\alpha}(X,Z)-\dim(X)$
only depends on $Z$ (see Proposition~\ref{dep_on_X}). 

In order to prove the basic properties of the minimal exponent for local complete intersections, we describe it 
as the minimal exponent of a hypersurface. Arguing locally, we may again assume that $Z$ has pure codimension $r$ in $X$
and it is defined in $X$ by the ideal
generated by $f_1,\ldots,f_r\in\cO_X(X)$. We consider $Y=X\times \A^r$, with coordinates $y_1,\ldots,y_r$ on $\A^r$, and let
$U=X\times\big(\A^r\smallsetminus \{0\}\big)$.

\begin{thm}\label{thm2_intro}
With the above notation, if $g=\sum_{i=1}^rf_iy_i\in\cO_Y(Y)$, then 
$$\widetilde{\alpha}(Z)=\widetilde{\alpha}(g\vert_U).$$
\end{thm}

The proof of Theorem~\ref{thm2_intro} relies on a general result of independent interest describing the $V$-filtration associated to $f_1,\ldots,f_d\in\cO_X(X)$
(without any complete intersection assumption) in terms of the microlocal $V$-filtration associated to $g=\sum_{i=1}^df_iy_i\in\cO_X(X)[y_1,\ldots,y_d]$;
see Theorem~\ref{thm_V_filtration} for the precise statement. 
Another application of this connection is a relation between $b$-functions corresponding to $f_1,\ldots,f_d$ and microlocal $b$-functions corresponding to $g$.
This greatly extends the main result of \cite{Mustata}, which says that the
Bernstein-Sato polynomial $b_Z(s)$ of $Z$ is equal to the reduced Bernstein-Sato polynomial $b_g(s)/(s+1)$ of $g$.

The description in Theorem~\ref{thm2_intro}, together with the results on minimal exponents of hypersurfaces from \cite{MP1}, allow
us to obtain similar results for local complete intersections. In order to state these, it is convenient to use a local version of the minimal exponent. 
If $Z$ is a local complete intersection in $X$ as above and $x\in Z$ is a point, then we put $\widetilde{\alpha}_x(Z):=\max_{V\ni x}\widetilde{\alpha}(V, Z\cap V)$,
where the maximum is over the open neighborhoods $V$ of $x$ in $X$.

\begin{thm}\label{thm3_intro}
Let $X$ be a smooth, irreducible, $n$-dimensional complex algebraic variety and 
let $Z$ be a local complete intersection closed subscheme of $X$, of pure codimension $r$. 
\begin{enumerate}
\item[i)] If $H$ is a smooth hypersurface in $X$ that contains no irreducible component of $Z$ and $Z_H=Z\cap H\hookrightarrow H$, then for every 
$x\in Z_H$, we have
$$\widetilde{\alpha}_x(H, Z_H)\leq\widetilde{\alpha}_x(X, Z).$$
\item[ii)] Given a smooth morphism $\mu\colon X\to T$ such that for every $t\in T$, $Z_t:=Z\cap \mu^{-1}(t)\hookrightarrow
X_t=\mu^{-1}(t)$ has pure codimension $r$, then the following hold: 
\begin{enumerate}
\item[${\rm ii_1)}$] For every $\alpha\in\Q_{>0}$, the set 
$$\big\{x\in Z\mid \widetilde{\alpha}_x(X_{\mu(x)}, Z_{\mu(x)})\geq \alpha\}$$
is open in $Z$. 
\item[${\rm ii_2)}$] There is an open subset $T_0$ of $T$ such that for every $t\in T_0$ and $x\in Z_t$, we have
$$\widetilde{\alpha}_x(X_t, Z_t)=\widetilde{\alpha}_x(X, Z).$$
\end{enumerate}
In particular, the set $\big\{\widetilde{\alpha}_x(X_{\mu(x)}, Z_{\mu(x)})\mid x\in Z\big\}$ is finite. Moreover, 
if $s\colon T\to X$ is a section of $\mu$ such that $s(T)\subseteq Z$, then the set 
$\big\{t\in T\mid \widetilde{\alpha}_{s(t)}(X_t, Z_t)\geq\alpha\big\}$ is open in $T$ for every $\alpha\in\Q_{>0}$.
\item[iii)] If $x\in Z$ is a point defined by the ideal $\frm_x$ and the ideal defining $Z$ at $x$ is contained in $\frm_x^k$, for some $k\geq 2$, then
$$\widetilde{\alpha}_x(Z)\leq \frac{n}{k}.$$
\end{enumerate}
\end{thm}

Another main result of the paper says that the minimal exponent controls the behavior of the Hodge filtration on local cohomology.
Recall that if $Z$ is locally a complete intersection of pure codimension $r$, the only nontrivial local cohomology of the structure sheaf is $\cH_Z^r(\cO_X)$,
and if $Z$ is defined by $f_1,\ldots,f_r$, then
$$\cH_Z^r(\cO_X)=\cO_X[1/f_1\cdots f_r]/\sum_{i=1}^r\cO_X[1/f_1\cdots\widehat{f_i}\cdots f_r].$$
The Hodge filtration on this mixed Hodge module was studied in \cite{MP2}. There is another natural filtration, the \emph{order filtration} (or \emph{Ext filtration}) given by
$$O_k\cH^r_Z(\cO_X)=\big\{u\in\cH_Z^r(\cO_X)\mid I_Z^{k+1}u=0\big\},$$
where $I_Z$ is the ideal defining $Z$. For every $k\geq 0$ we have $F_k\cH^r_Z(\cO_X)\subseteq O_k\cH^r_Z(\cO_X)$ 
and if equality holds for $k=p$, then it holds for all $k$, with $0\leq k\leq p$.
The \emph{singularity level} $p(Z)$ of the Hodge filtration on $\cH^r_Z(\cO_X)$ is 
$$p(Z)={\rm sup}\big\{k\geq 0\mid F_k\cH^r_Z(\cO_X)=O_k\cH^r(\cO_X)\big\},$$
with the convention that this is $-1$ if the above set is empty. With this notation, we prove

\begin{thm}\label{thm1_intro}
If $X$ is a smooth, irreducible, complex algebraic variety and
$Z$ is a local complete intersection closed subscheme of $X$, of pure codimension $r$, then
$$p(Z)=\max\big\{\lfloor\widetilde{\alpha}(Z)\rfloor-r,-1\big\}.$$
\end{thm}

In particular, by combining Theorems~\ref{thm1_intro} and \ref{thm3_intro}, we see that the invariant $p(Z)$ satisfies analogous properties
to those in Theorem~\ref{thm3_intro}. This was already shown in \cite[Section~9]{MP2} by different methods. 
For an application of Theorem~\ref{thm1_intro} to an Inversion-of-Adjunction type statement, see Corollary~\ref{cor_IA}.
The main ingredient in the proof of Theorem~\ref{thm1_intro} is the description of the Hodge filtration on $\cH_Z^r(\cO_X)$ in terms of the
$V$-filtration on $B_{\mathbf f}$. This relies on the interplay between the Hodge filtration and the $V$-filtration for filtered $\cD$-modules 
that underlie mixed Hodge modules. In the case of one function, this is built into the definition of Hodge modules. 
However, the case of several functions is more subtle and has only recently been elucidated in \cite{CD}. Using these results, we give
the following description of the Hodge filtration on local cohomology:

\begin{thm}\label{thm1.5_intro}
If $X$ is a smooth, irreducible, complex algebraic variety and $f_1,\ldots,f_r\in\cO_X(X)$ define a complete intersection closed subscheme $Z$
of pure codimension $r$, then for every $p\geq 0$, we have
$$F_p\cH^r_Z(\cO_X)=\left\{\left[\sum_{|\alpha|\leq p}\frac{\alpha_1!\cdots\alpha_r! h_{\alpha}}{f_1^{\alpha_1+1}\cdots f_r^{\alpha_r+1}}\right]\mid
\sum_{|\alpha|\leq p}h_{\alpha}\partial_t^{\alpha}\delta_{\mathbf f}\in V^rB_{\mathbf f}\right\}.$$
\end{thm}

One interesting question that remains open is the precise relation between $\widetilde{\alpha}(Z)$ and the Bernstein-Sato polynomial of $Z$. We recall that for an arbitrary closed subscheme of the smooth variety $X$,
one can define a Bernstein-Sato polynomial $b_Z(s)\in\Q[s]$, extending the classical notion from the case of hypersurfaces (see \cite{BMS}). As in the classical case, all
its roots are negative rational numbers, with the largest root being $-\lct(X,Z)$.
It is easy to see that if $Z$
is a (nonempty) local complete intersection of pure codimension $r$, then $(s+r)$ divides $b_Z(s)$, see Proposition~\ref{prop_r} below. By analogy with the definition of the minimal exponent in the case of hypersurfaces, we define $\widetilde{\gamma}(Z)$ to be the negative of the largest root of $b_Z(s)/(s+r)$ (with the convention that this is infinite if $b_Z(s)/(s+r)=1$). 

\begin{question}\label{question_gamma}
If $Z$ is locally a complete intersection in the smooth irreducible variety $X$, of pure codimension $r$, do we have $\widetilde{\alpha}(Z)=\widetilde{\gamma}(Z)$? 
\end{question}

Note that in light of Theorem~\ref{thm1_intro}, a positive answer to Question~\ref{question_gamma} would provide a positive answer to
\cite[Conjecture~9.11]{MP2}, relating the Hodge filtration on $\cH_Z^r(\cO_X)$ and the invariant $\widetilde{\gamma}(Z)$.
We can prove the following relation between the two invariants:

\begin{thm}\label{thm4_intro}
With the notation in Question~\ref{question_gamma}, we have $\widetilde{\alpha}(Z)\geq\widetilde{\gamma}(Z)$ and
$$\min\big\{\widetilde{\alpha}(Z),r+1\big\}=\min\big\{\widetilde{\gamma}(Z),r+1\big\}.$$
\end{thm}

We recall that by \cite[Theorem~4]{BMS}, under the assumptions of Theorem~\ref{thm4_intro}, the subscheme $Z$ has rational singularities if and only if
$\widetilde{\gamma}(Z)>r$. By combining Theorems~\ref{thm1_intro} and \ref{thm4_intro}, we obtain the following result, which gives a positive answer to 
\cite[Conjecture~8.4]{MP2}.

\begin{cor}\label{cor_intro}
If $X$ is a smooth, irreducible variety and
$Z$ is a local complete intersection closed subscheme of $X$, of pure codimension $r$, 
then $Z$ has rational singularities if and only if $\widetilde{\alpha}(Z)>r$. In particular, if
$F_1\cH^r_Z(\cO_X)=O_1\cH^r_Z(\cO_X)$, then
$Z$ has rational singularities. 
\end{cor}

We note that since the first version of this paper was written, a positive answer to Question~\ref{question_gamma} has been given in \cite{Dirks}.

\medskip

\noindent {\bf Outline of the paper}. In Section~2 we review the basic facts about $V$-filtrations and $b$-functions. The following section is devoted to the result
relating the $V$-filtration associated to $f_1,\ldots,f_d$ and the microlocal $V$-filtration associated to $\sum_{i=1}^df_iy_i$. In Section~4 we introduce the minimal exponent
of a local complete intersection subscheme, prove the description in Theorem~\ref{thm2_intro}, as well as various general properties of this invariant, including the
ones in Theorem~\ref{thm3_intro}. In Section~5 we relate the minimal exponent to the Hodge filtration on local cohomology, proving Theorems~\ref{thm1_intro}
and \ref{thm1.5_intro}.
Finally, in the last section, we discuss the connection with the Bernstein-Sato polynomial and prove Theorem~\ref{thm4_intro}.

\noindent {\bf Acknowledgments}. We would like to thank Mihnea Popa and Christian Schnell for many discussions related to the subject of this work. 
We are also grateful to Karl Schwede for providing some useful references and to the anonymous referee for the comments on a previous version of this article.

\section{Review of $V$-filtrations}\label{section_review}

In this section we recall the definition and some basic properties of $V$-filtrations. 
For details, we refer to \cite{Kashiwara}, \cite[Section~1]{BMS}, and \cite[Section~3.1]{Saito-MHP}.
Let $X$ be a fixed smooth, irreducible, complex algebraic variety.
Recall that $\cD_X$ denotes the sheaf of differential operators on $X$. In this paper all $\cD$-modules will be \emph{left} $\cD$-modules.
For general facts about $\cD$-modules, we refer to \cite{HTT}. 

Given nonzero regular functions $f_1,\ldots,f_d\in\cO_X(X)$, we denote by $\fra\subseteq\cO_X$ the ideal $(f_1,\ldots,f_d)$ and by $Z$
the closed subscheme of $X$ defined by $\fra$.
We consider the graph embedding 
$$\iota\colon X\hookrightarrow W=X\times \A^d,\quad \iota(x)=\big(x,f_1(x),\ldots,f_d(x)\big)$$
and the $\cD$-module theoretic push-forward $B_{\mathbf f}=\iota_+(\cO_X)$ (we denote by ${\mathbf f}$ the $d$-tuple $(f_1,\ldots,f_d)$). 
We denote the standard coordinates on $\A^d$ by $t_1,\ldots,t_d$ and use multi-index notation, so for $\beta=(\beta_1,\ldots,\beta_d)\in\Z_{\geq 0}^d$,
we put $t^{\beta}=t_1^{\beta_1}\cdots t_d^{\beta_d}$ and $\partial_t^{\beta}=\partial_{t_1}^{\beta_1}\cdots\partial_{t_d}^{\beta_d}$. 
We also put $\beta!=\prod_i\beta_i!$ and $|\beta|=\sum_i\beta_i$. Finally, we consider $s_i=-\partial_{t_i}t_i$ for $1\leq i\leq d$ and $s=\sum_{i=1}^ds_i$. 

It is convenient to consider $B_{\mathbf f}$ as an $\cR$-module on $X$, where 
$\cR=\cD_X\langle t_1,\ldots,t_d,\partial_{t_1},\ldots\partial_{t_d}\rangle$.
The general description
of $\cD$-module push-forward via closed immersions gives
$$B_{\mathbf f}=\bigoplus_{\beta\in\Z_{\geq 0}^d}\cO_X\partial_t^{\beta}\delta_{\mathbf f},$$
where the actions of $\cO_X$ and $\partial_{t_i}$ are the obvious ones, while the actions of $D\in {\rm Der}_{\C}(\cO_X)$ and of the $t_i$ are given by
\begin{equation}\label{eq_action_B_f}
D\cdot h\partial_t^{\beta}\delta_{\mathbf f}=D(h)\partial_t^{\beta}\delta_{\mathbf f}-\sum_{i=1}^dD(f_i)h\partial_t^{\beta+e_i}\delta_{\mathbf f}\quad\text{and}\quad
t_i\cdot h\partial_t^{\beta}\delta_{\mathbf f}=f_ih\partial_t^{\beta}\delta_{\mathbf f}-\beta_ih\partial_t^{\beta-e_i}\delta_{\mathbf f},
\end{equation}
where $e_1,\ldots,e_d$ is the standard basis of $\Z^d$. 
We will also consider on $B_{\bff}$ the \emph{Hodge filtration}\footnote{It is often the case that one shifts this filtration so that what we denote by $F_pB_{\bff}$
is considered to be $F_{p+d}B_{\bff}$.} given by
$$F_pB_{\bff}=\bigoplus_{|\beta|\leq p}\cO_X\partial_t^{\beta}\delta_{\bff}.$$

It is sometimes convenient to consider the larger $\cR$-module $B_{\mathbf f}^+$ corresponding to the push-forward
$\iota_+\big(\cO_X[1/f_1\cdots f_d]\big)$, namely
$$B_{\mathbf f}^+=\bigoplus_{\beta\in\Z_{\geq 0}^d}\cO_X[1/f_1\cdots f_d]\partial_t^{\beta}\delta_{\mathbf f}.$$
We may also write the elements of $B_{\mathbf f}^+$ in terms of the operators $s_1,\ldots,s_d$, as follows. 
Let us put $Q_m(x)=(-1)^m m!{x\choose m}=x(x-1)\cdots (x-m+1)\in\C[x]$ for $m\in\Z_{\geq 0}$ and 
$$Q_{\beta}(s_1,\ldots,s_d)=\prod_{i=1}^dQ_{\beta_i}(s_i)\in {\mathbf C}[s_1,\ldots,s_d]\quad\text{for}\quad\beta=(\beta_1,\ldots,\beta_d)\in\Z_{\geq 0}^d.$$ 
It follows from 
Lemma~\ref{lem0_appendix} that we can write $\partial_{t_i}^{\beta_i}=Q_{\beta_i}(s_i)t_i^{-\beta_i}$ and since $t_i^{-\beta_i}\delta_{\mathbf f}=\tfrac{1}{f_i^{\beta_i}}\delta_{\mathbf f}$
in $B_{\mathbf f}^+$,
we have
\begin{equation}\label{eq_in_terms_s}
\sum_{\beta}h_{\beta}\partial_t^{\beta}\delta_{\mathbf f}=\sum_{\beta}h_{\beta}\tfrac{Q_{\beta}(s_1,\ldots,s_d)}{f_1^{\beta_1}\cdots f_d^{\beta_d}}\delta_{\mathbf f},
\end{equation}
where $h_{\beta}\in\cO_X[1/f_1\cdots f_d]$ for all $\beta\in\Z_{\geq 0}^d$ (with only finitely many nonzero).

Since the polynomials $Q_{\beta}(s_1,\ldots,s_d)$, with $\beta=(\beta_1,\ldots,\beta_d)$ running over $\Z^d_{\geq 0}$, give a basis of ${\mathbf C}[s_1,\ldots,s_d]$ over ${\mathbf C}$, 
it is easy to see that if we put ${\mathbf f}^{\mathbf s}=f_1^{s_1}\cdots f_d^{s_d}$, then 
we have an isomorphism of ${\mathcal R}$-modules 
\begin{equation}\label{eq_Phi}
B_{\mathbf f}^+\simeq \cO_X[1/f_1\cdots f_d,s_1,\ldots,s_d]{\mathbf f}^{\mathbf s},
\end{equation}
that maps $\partial_t^{\beta}\delta_{\mathbf f}$ to $\tfrac{Q_{\beta}(s_1,\ldots,s_r)}{f_1^{\beta_1}\cdots f_r^{\beta_r}}{\mathbf f}^{\mathbf s}$.
Note that a derivation $D\in {\rm Der}_{\C}(\cO_X)$ acts on ${\mathbf f}^{\mathbf s}$
in the expected way:
$$D\cdot {\mathbf f}^{\mathbf s}=\sum_{i=1}^d\frac{s_iD(f_i)}{f_i}{\mathbf f}^{\mathbf s}.$$
We also note that the action of $t_i$ on the left-hand side of (\ref{eq_Phi}) corresponds on the right-hand side to the automorphism that maps $s_i$ to $s_i+1$,
and similarly, the action of $s_i=-\partial_{t_i}t_i$ on the left-hand side of (\ref{eq_Phi}) corresponds on the right-hand side to multiplication by $s_i$. 
We sometimes tacitly use this isomorphism to denote an element of $B_{\mathbf f}^+$ by $P(s_1,\ldots,s_d){\mathbf f}^{\mathbf s}$,
for some $P\in\cO_X[1/f_1\cdots,f_d,s_1,\ldots,s_d]$. Note that it follows from (\ref{eq_in_terms_s}) that if we write 
$P(s_1,\ldots,s_d)=\sum_{\beta}g_{\beta}Q_{\beta}(s_1,\ldots,s_d)$, with $g_{\beta}\in\cO_X[1/f_1\cdots f_d]$, then
$P(s_1,\ldots,s_d){\mathbf f}^{\mathbf s}\in B_{\mathbf f}$ if and only if $g_{\beta}\in\cO_X\cdot \tfrac{1}{f_1^{\beta_1}\cdots f_d^{\beta_d}}$ for all $\beta\in\Z_{\geq 0}^d$.

We now turn to the $V$-filtration.
On $\cR$ we consider the decreasing filtration
$$V^m\cR=\bigoplus_{|\alpha|-|\beta|\geq m}\cD_Xt^{\alpha}\partial_t^{\beta}$$
for $m\in\Z$. It is then clear that
$$V^0\cR=\cD_X\langle t_i, t_i\partial_{t_j}\mid i,j\rangle\quad\text{and}\quad V^1\cR=\sum_{i=1}^dt_i\cdot V^0\cR=\sum_{i=1}^dV^0\cR\cdot t_i.$$

The $V$-filtration on $B_{\bff}$ is a decreasing, exhaustive filtration $(V^{\gamma}B_{\mathbf f})_{\gamma}$ indexed by rational numbers, which is \emph{discrete and left-continuous\footnote{This means that there is a positive integer $\ell$ such that $V^{\gamma}B_{\mathbf f}$ has constant value for all 
$\gamma$ in an interval of the form $\left(\tfrac{i-1}{\ell},\tfrac{i}{\ell}\right]$, with $i\in\Z$.}} and satisfies the following properties:
\begin{enumerate}
\item[i)] Every $V^{\gamma}B_{\bff}$ is a coherent $V^0\cR$-submodule of $B_{\bff}$.
\item[ii)] $t_i\cdot V^{\gamma}B_{\bff}\subseteq V^{\gamma+1}B_{\bff}$ and $\partial_{t_i}\cdot V^{\gamma}B_{\bff}\subseteq V^{\gamma-1}B_{\bff}$ for all $i\leq d$ and $\gamma\in\Q$. 
\item[iii)] $V^1\cR\cdot V^{\gamma}B_{\bff}=V^{\gamma+1}B_{\bff}$ if $\gamma\gg 0$.
\item[iv)] The action of $s+\gamma$ on ${\rm Gr}_V^{\gamma}(B_{\bff})$ is nilpotent for all $\gamma\in\Q$.
\end{enumerate}
Here we put ${\rm Gr}_V^{\gamma}(B_{\bff})=V^{\gamma}B_{\bff}/V^{>\gamma}B_{\bff}$, where $V^{>\gamma}B_{\bff}=\bigcup_{\beta>\gamma}V^{\beta}B_{\bff}$. 

By the theory of Kashiwara \cite{Kashiwara}, extending a result of Malgrange \cite{Malgrange}, there is a unique such $V$-filtration.
Uniqueness follows by easy arguments, while existence is a deeper statement. 

\begin{rmk}\label{restr_complement_Z}
We note that for every $\gamma\in\Q$, we have $V^{\gamma}B_{\bff}\vert_{X\smallsetminus Z}=B_{\bff}\vert_{X\smallsetminus Z}$. 
\end{rmk}

We recall that 
the $V$-filtration on $B_{\bff}$ induces on $\cO_X\simeq\cO_X\delta_{\bff}$ the filtration by the multiplier ideals of $\fra$
(for the definition and basic properties of multiplier ideals, we refer to \cite[Section~9]{Lazarsfeld}). More precisely, it follows
from \cite[Theorem~1]{BMS} that for every $\gamma\in\Q_{>0}$, we have
$$\{h\in\cO_X\mid h\delta_{\bff} \in V^{\gamma}B_{\bff}\}=\cJ(\fra^{\gamma-\epsilon}),\quad \text{for}\quad 0<\epsilon\ll 1.$$
In particular, we have 
\begin{equation}\label{formula_lct}
\delta_{\bff}\in V^{\gamma}B_{\bff} \quad\text{if and only if}\quad \gamma\leq\lct(\fra),
\end{equation}
where $\lct(\fra)$ is the log canonical threshold of $\fra$, characterized as $\min\{\lambda>0\mid \cJ(\fra^{\lambda})\neq\cO_X\}$
(we also denote this by $\lct(X,Z)$).

The existence of $V$-filtrations is closely related to the existence of $b$-functions. Recall that for every $u\in B_{\bff}$, the $b$-function
$b_u(s)$ of $u$ is the monic generator of the ideal
\begin{equation}\label{defi_b_fcn}
\big\{b(s)\in\C[s]\mid b(s)u\in V^1\cR\cdot u\big\}.
\end{equation}
We note that the condition in (\ref{defi_b_fcn}) is equivalent to $b(s)V^0\cR\cdot u\subseteq V^1\cR\cdot u$: this follows easily from Lemmas~\ref{lem2_appendix}
and \ref{lem3_appendix} in the Appendix and the fact that for all $i$ and $j$, we have $t_i\partial_{t_j}\cdot V^0\cR\subseteq V^0\cR$ and $t_i\cdot V^0\cR\subseteq V^1\cR$.
It follows from the results in \cite{Kashiwara} (see also \cite{BMS}) that for every $u\in B_{\bff}$, the ideal (\ref{defi_b_fcn}) is nonzero and thus $b_u(s)$ is well-defined. Moreover,
all its roots are rational. The $V$-filtration can then
 be described as
\begin{equation}\label{eq_formula_V_filtration}
V^{\gamma}B_{\bff}=\big\{u\in B_{\bff}\mid \text{all roots of}\,b_u(s)\,\text{are}\,\leq -\gamma\}.
\end{equation}
In particular, for $u=\delta_{\mathbf f}\in B_{\mathbf f}$, the $b$-function $b_u(s)$ is the Bernstein-Sato polynomial of the ideal $\fra$, introduced and studied in 
\cite{BMS}; this only depends on $\fra$ (not on the choice of $f_1,\ldots,f_d$) and we denote it by $b_Z(s)$.
In the case $d=1$ and $f=f_1$, this is the $b$-function of a hypersurface, introduced independently by Bernstein and Sato; we also denote it by $b_f(s)$.
Note that for any $d$, it follows from (\ref{formula_lct}) and (\ref{eq_formula_V_filtration}) that 
\begin{equation}\label{eq_lct_b_function}
\max\{\lambda\in\Q\mid b_Z(\lambda)=0\}=-\lct(X,Z). 
\end{equation}

\medskip

Suppose now that $d=1$, so we have only one nonzero regular function, that we denote $f$. In this case, Saito introduced in \cite{Saito_microlocal} the 
\emph{microlocal $V$-filtration} associated to $f$, defined as follows. Instead of $B_f$, we consider
$$\widetilde{B}_f:=\bigoplus_{j\in\Z}\cO_X\partial_t^j\delta_f,$$
which is a left module over $\widetilde{\cR}=\cD_X\langle t,\partial_t,\partial_t^{-1}\rangle$. Note that the relation $[\partial_t,t]=1$ implies 
$[\partial_t^{-1},t]=-\partial_t^{-2}$. The action of $\cO_X$ and of $\partial_t$, $\partial_t^{-1}$ on $\widetilde{B}_f$ are the obvious ones,
while the action of derivations and of $t$ are given by the following analogue of (\ref{eq_action_B_f}): for every $D\in {\rm Der}_{\C}(\cO_X)$, $h\in\cO_X$, and $j\in\Z$, we have
\begin{equation}\label{eq_action_B_f2}
D\cdot h\partial_t^j\delta_f=D(h)\partial_t^j\delta_f-hD(f)\partial_t^{j+1}\delta_f\quad\text{and}\quad t\cdot h\partial_t^j\delta_f=fh\partial_t^j\delta_f-jh\partial_t^{j-1}\delta_f.
\end{equation}
The $V$-filtration on $\widetilde{\cR}$ is defined as before: for $m\in\Z$, we have
$$V^m\widetilde{\cR}=\bigoplus_{i-j\geq m}\cD_Xt^i\partial_t^j,$$
where this time $i\in\Z_{\geq 0}$ and $j\in\Z$. It is easy to see that
$$V^0\widetilde{\cR}=\cD_X\langle t,t\partial_t,\partial_t^{-1}\rangle\quad\text{and}\quad V^j\widetilde{\cR}=\partial_t^{-j}\cdot V^0\widetilde{\cR}= V^0\widetilde{\cR}\cdot \partial_t^{-j}\,\,\text{for all}\,\,j\in\Z.$$

The Hodge filtration on $\widetilde{B}_f$ is given by 
$$F_p\widetilde{B}_f=\bigoplus_{i\leq p}\cO_X\partial_t^i\delta_f.$$
On the other hand, the microlocal $V$-filtration on $\widetilde{B}_f$ is given by 
$$V^{\gamma}\widetilde{B}_f=V^{\gamma}B_f\oplus\bigoplus_{j\geq 1}\cO_X\partial_t^{-j}\delta_f\quad\text{for}\quad \gamma\leq 1$$
and 
$$V^{\gamma}\widetilde{B}_f=\partial_t^{-j}V^{\gamma-j}\widetilde{B}_f\quad \text{for}\quad \gamma>1,$$
where $j\in\Z$ is such that $0<\gamma-j\leq 1$. The following properties follow easily from the properties of the
$V$-filtration on $B_f$:
\begin{enumerate}
\item[i)] $V^1\widetilde{\cR}\cdot V^{\gamma}\widetilde{B}_f\subseteq V^{\gamma+1}\widetilde{B}_f$ for all $\gamma\in\Q$.
\item[ii)] Every $V^{\gamma}\widetilde{B}_f$ is a finitely generated $V^0\widetilde{\cR}$-module and it generates $\widetilde{B}_f$
over $\widetilde{\cR}$.
\item[iii)] $s+\gamma$ is nilpotent on ${\rm Gr}_V^{\gamma}(\widetilde{B}_f)$ for every $\gamma\in\Q$. 
\end{enumerate}
A useful property of the microlocal $V$-filtration is that for every $j\in\Z$ and every $\gamma\in\Q$, multiplication by $\partial_t^j$
gives an isomorphism
\begin{equation}\label{eq_mult_partial}
\partial_t^j\colon V^{\gamma}\widetilde{B}_f\simeq V^{\gamma-j}\widetilde{B}_f
\end{equation}
(see \cite[Lemma~2.2]{Saito_microlocal}). 

Given $u\in \widetilde{B}_f$, the microlocal $b$-function $\widetilde{b}_u(s)\in\C[s]$ is the monic generator of the ideal
$$\big\{b(s)\in\C[s]\mid b(s)u\in V^1\widetilde{\cR}\cdot u\big\}$$
(the fact that this ideal is nonzero and all roots of $\widetilde{b}_u(s)$ are rational, follows from the fact that
$V^{\gamma}\widetilde{B}_f\subseteq V^1\widetilde{\cR}\cdot u$ for $\gamma\gg 0$). 
We have the following analogue of (\ref{eq_formula_V_filtration}) describing the microlocal $V$-filtration in terms of
microlocal $b$-functions:
\begin{equation}\label{eq2_formula_V_filtration}
V^{\gamma}\widetilde{B}_{f}=\big\{u\in \widetilde{B}_{f}\mid \text{all roots of}\,\,\widetilde{b}_u(s)\,\,\text{are}\,\leq -\gamma\}.
\end{equation}
An important example is that when
$u=\delta_f\in \widetilde{B}_f$, when we write $\widetilde{b}_f(s)$ for $\widetilde{b}_{\delta_f}(s)$.
If $f$ is not invertible, then it is easy to see that $b_f(s)=b_{\delta_f}(s)$ is divisible by $(s+1)$, and in fact we have
$$\widetilde{b}_f(s)=b_f(s)/(s+1)$$
(see \cite[Proposition~0.3]{Saito_microlocal}). 

Under the same assumption that $f$ is not invertible, the negative of the largest root of $b_f(s)/(s+1)$ is the 
\emph{minimal exponent} $\widetilde{\alpha}(f)$, that we also write as $\widetilde{\alpha}(H)$ if $H$ is the hypersurface defined by $f$.
Here we make the convention that $\widetilde{\alpha}(f)=\infty$ if $b_f(s)/(s+1)=1$. Note that by 
(\ref{eq_lct_b_function}), we have $\min\big\{\widetilde{\alpha}(H),1\big\}=\lct(X, H)$. 
We also recall that by a result of Saito (see \cite[Theorem~0.4]{Saito-B}), we have $\widetilde{\alpha}(H)>1$ if and only if $H$ has rational singularities. 
For a discussion of minimal exponents and basic properties, see \cite[Section~6]{MP1}. One property that is very relevant for us is its connection with the $V$-filtration:
it follows from (\ref{eq2_formula_V_filtration}) and the fact that $b_f(s)/(s+1)=\widetilde{b}_f(s)$ that
\begin{equation}\label{eq_MLCT}
\widetilde{\alpha}(H)={\rm sup}\{\gamma\in\Q_{\geq 0}\mid \delta\in V^{\gamma}\widetilde{B}_f\}.
\end{equation}
In terms of the $V$-filtration on $B_f$, this is equivalent to the fact that for every nonnegative integer $q$ and every rational number
$\gamma\in (0,1]$, we have
\begin{equation}\label{eq_char_min_exp}
\widetilde{\alpha}(H)\geq q+\gamma\quad\text{if and only if}\quad \partial_t^q\delta_f\in V^{\gamma}B_f
\end{equation}
(see \cite{Saito-MLCT}). 

We will also make use of a local version of the minimal exponent of hypersurfaces. If $f\in\cO_X(X)$ and $H$ are as above and $x\in H$, then
$$\widetilde{\alpha}_x(f):=\max_{U\ni x}\widetilde{\alpha}(f\vert_U),$$
where the maximum is over all open neighborhoods $U$ of $x$. With this notation, we have
$$\widetilde{\alpha}(f)=\min_{x\in H}\widetilde{\alpha}_x(f).$$

\bigskip

\section{General $V$-filtrations via microlocal $V$-filtrations along hypersurfaces}\label{section_rel_Vfilt}

Let $X$ be a smooth, irreducible, complex algebraic variety. In this section, we consider nonzero regular functions
$f_1,\ldots,f_d\in\cO_X(X)$ and let $g=\sum_{i=1}^df_iy_i\in\cO_Y(Y)$, where $Y=X\times\A^d$ and we denote by 
$y_1,\ldots,y_d$ the standard coordinates on $\A^d$. Our goal is to relate the $V$-filtration on 
$B_{\bff}=\bigoplus_{\alpha\in\Z_{\geq 0}^d}\cO_X\partial_t^{\alpha}\delta_{\bff}$ and the microlocal $V$-filtration on
$\widetilde{B}_g=\bigoplus_{j\in\Z}\cO_Y\partial_z^j\delta_g$ (note that we denote by $z$ the
extra variable that acts on $\widetilde{B}_g$ in order to avoid confusion with the variables $t_1,\ldots,t_d$ that act on $B_{\bff}$).

Note that $g$ is homogeneous of degree $1$ with respect to the grading on $\cO_Y$ such that $\cO_X$ lies in degree $0$ and
${\rm deg}(y_i)=1$ for all $i$. We have a corresponding grading on $\cD_Y$ such that ${\rm deg}(\partial_{y_i})=-1$ for all $i$. 
Furthermore, on $\cD_Y\langle z,\partial_z,\partial_z^{-1}\rangle$ we have a grading such that ${\rm deg}(z)=1={\rm deg}(\partial_z^{-1})$
and ${\rm deg}(\partial_z)=-1$.

We can write
$$\widetilde{B}_g=\bigoplus_{j\in\Z}\bigoplus_{\alpha\in\Z_{\geq 0}^d}\cO_Xy^{\alpha}\partial_z^j\delta_g.$$
If for every $m\in\Z$ we put 
$$\widetilde{B}_g^{(m)}=\bigoplus_{\alpha\in\Z_{\geq 0}^d}\cO_Xy^{\alpha}\partial_z^{|\alpha|-m}\delta_g,$$
then it follows easily from the formulas (\ref{eq_action_B_f2}) 
and the fact that $g$ is homogeneous of degree $1$ 
that the decomposition $\widetilde{B}_g=\bigoplus_{m\in\Z}\widetilde{B}_g^{(m)}$
makes $\widetilde{B}_g$ a graded $\cD_Y\langle z,\partial_z,\partial_z^{-1}\rangle$-module. 
Let $\theta_y:=\sum_{i=1}^dy_i\partial_{y_i}\in\cD_Y$.

\begin{lem}\label{lem_theta}
For every $m\in\Z$ and every $u\in \widetilde{B}_g^{(m)}$, we have $(\theta_y-s)u=mu$.
\end{lem}

\begin{proof}
We may and will assume that $u=hy^{\alpha}\partial_z^{|\alpha|-m}\delta_g$, for some $h\in\cO_X$.
On one hand, we have
$$\theta_y\cdot u=\sum_{i=1}^d\alpha_i y_ihy^{\alpha-e_i}\partial_z^{|\alpha|-m}\delta_g-\sum_{i=1}^dy_if_ihy^{\alpha}\partial_z^{|\alpha|-m+1}\delta_g
=|\alpha|hy^{\alpha}\partial_z^{|\alpha|-m}\delta_g-ghy^{\alpha}\partial_z^{|\alpha|-m+1}\delta_g.$$
On the other hand, we have
$$su=-\partial_zzu=-ghy^{\alpha}\partial_z^{|\alpha|-m+1}\delta_g+\big(|\alpha|-m\big)hy^{\alpha}\partial_z^{|\alpha|-m}\delta_g,$$
and the formula in the lemma follows. 
\end{proof}

Note that $V^{\gamma}\widetilde{B}_g$ is preserved by the action of $\theta_y-s$ for every $\gamma\in\Q$. By Lemma~\ref{lem_theta},
the decomposition $\widetilde{B}_g=\bigoplus_{m\in\Z}\widetilde{B}_g^{(m)}$ is an eigenspace decomposition with respect to the endomorphism
$\theta_y-s$. We deduce that we get an induced decomposition
$$V^{\gamma}\widetilde{B}_g=\bigoplus_{m\in\Z}V^{\gamma}\widetilde{B}_g^{(m)},\quad\text{where}\quad V^{\gamma}\widetilde{B}_g^{(m)}=
V^{\gamma}\widetilde{B}_g\cap \widetilde{B}_g^{(m)}.$$
We get a corresponding decomposition
$${\rm Gr}_V^{\gamma}(\widetilde{B}_g)=\bigoplus_{m\in\Z}{\rm Gr}_V^{\gamma}(\widetilde{B}_g^{(m)}),\quad\text{where}\quad
{\rm Gr}_V^{\gamma}(\widetilde{B}_g^{(m)})=V^{\gamma}\widetilde{B}_g^{(m)}/V^{>\gamma}\widetilde{B}_g^{(m)}.$$
Finally, we note that it follows from (\ref{eq_mult_partial}) that 
\begin{equation}\label{eq_comp_V_fil}
V^{\gamma}\widetilde{B}_g^{(m)}=\partial_z^{-m}V^{\gamma-m}\widetilde{B}_g^{(0)}\quad\text{for all}\quad \gamma\in\Q, m\in\Z.
\end{equation}

We now define the map that will allow us to compare the $V$-filtration on $B_{\bff}$ with the microlocal $V$-filtration on $\widetilde{B}_g$.
Let $\varphi\colon \widetilde{B}_g\to B_{\bff}$ be the unique $\cO_X$-linear map such that
\begin{equation}\label{eq_formula_phi}
\varphi(y^{\alpha}\partial_z^j\delta_g)=\partial_t^{\alpha}\delta_{\bff}\quad\text{for all}\quad \alpha\in\Z_{\geq 0}^d, j\in\Z.
\end{equation}
It is clear from the definition that for every $m\in\Z$, $\varphi$ induces an isomorphism of $\cO_X$-modules $\widetilde{B}_g^{(m)}\simeq B_{\bff}$. 
We collect in the following proposition some basic properties of $\varphi$. 

\begin{prop}\label{prop_properties_phi}
With the above notation, the following hold:
\begin{enumerate}
\item[i)] The map $\varphi$ is $\cD_X$-linear.
\item[ii)] We have $\varphi(\partial_zu)=\varphi(u)=\varphi(\partial_z^{-1}u)$ for every $u\in \widetilde{B}_g$.
\item[iii)] We have $\varphi(y_iu)=\partial_{t_i}\varphi(u)$ for every $u\in\widetilde{B}_g$ and $1\leq i\leq d$. 
\item[iv)] We have $\varphi(\partial_{y_i}u)=-t_i\varphi(u)$ for every $u\in \widetilde{B}_g$ and $1\leq i\leq d$.
\item[v)] We have $\varphi(su)=(s-m)\varphi(u)$ for every $u\in \widetilde{B}^{(m)}_g$, where $m\in\Z$.
\end{enumerate}
\end{prop}

\begin{proof}
For i), since $\varphi$ is $\cO_X$-linear by definition, it is enough to show that $\varphi(Du)=D\varphi(u)$ for every $D\in {\rm Der}_{\C}(\cO_X)$
and every $u\in \widetilde{B}_g$. We may and will assume that $u=hy^{\alpha}\partial_z^j\delta_g$ for some $h\in\cO_X$, $\alpha\in\Z_{\geq 0}^d$ and $j\in\Z$.
In this case, we have
$$\varphi(Du)=\varphi\big(D(h)y^{\alpha}\partial_z^j\delta_g-hD(g)y^{\alpha}\partial_z^{j+1}\delta_g\big)=
\varphi\big(D(h)y^{\alpha}\partial_z^j\delta_g-\sum_{i=1}^dhD(f_i)y^{\alpha+e_i}\partial_z^{j+1}\delta_g\big)$$
$$=D(h)\partial_t^{\alpha}\delta_{\bff}-\sum_{i=1}^dhD(f_i)\partial_t^{\alpha+e_i}\delta_{\bff}=D\cdot h\partial_t^{\alpha}\delta_{\bff}=D\varphi(u).$$
This completes the proof of i).

The assertions in ii) and iii) follow directly from the definition. In order to prove iv), we may assume that $u=hy^{\alpha}\partial_z^j\delta_g$ for some $h\in\cO_X$, 
$\alpha\in\Z_{\geq 0}^d$ and $j\in\Z$. We then have
$$\varphi(\partial_{y_i}u)=\varphi\big(h\alpha_iy^{\alpha-e_i}\partial_z^j\delta_g-f_ihy^{\alpha}\partial_z^{j+1}\delta_g\big)$$
$$=\alpha_ih\partial_t^{\alpha-e_i}\delta_{\bff}-f_ih\partial_t^{\alpha}\delta_{\bff}=-t_i\cdot h\partial_t^{\alpha}\delta_{\bff}=-t_i\varphi(u).$$
We thus obtain the assertion in iv).

Finally, in order to prove v), we note that by Lemma~\ref{lem_theta}, if $u\in \widetilde{B}^{(m)}_g$, then $su=(\theta_y-m)u$, hence using
iii) and iv), we have
$$\varphi(su)=\varphi\big((\theta_y-m)u\big)=\sum_{i=1}^d\varphi(y_i\partial_{y_i}u)-m\varphi(u)=-\sum_{i=1}^d\partial_{t_i}t_i\varphi(u)-m\varphi(u)=
(s-m)\varphi(u).$$
\end{proof}

We now come to the main result of this section.
Let us denote by $\varphi_0$ the restriction of $\varphi$ to $\widetilde{B}_g^{(0)}$. Note that since $\varphi_0$ is bijective, the assertion in the next theorem
together with (\ref{eq_comp_V_fil}) say that the $V$-filtration on $B_{\bff}$ and the microlocal $V$-filtration on $\widetilde{B}_g$ determine each other.

\begin{thm}\label{thm_V_filtration}
With the above notation, for every $\gamma\in\Q$, we have
$$V^{\gamma}\widetilde{B}_g^{(0)}
=\varphi_0^{-1}(V^{\gamma}B_{\bff}).$$
\end{thm}

\begin{proof}
The argument is similar to that proving the uniqueness of $V$-filtrations (see for example \cite[Lemme~3.1.2]{Saito-MHP}). 
Recall that we write 
$${\mathcal R}=\cD_X\langle t_1,\ldots,t_d,\partial_{t_1},\ldots,\partial_{t_d}\rangle\quad\text{and}\quad \widetilde{\mathcal R}=\cD_Y\langle z,\partial_z,\partial_z^{-1}\rangle.$$
Let's prove first the inclusion
\begin{equation}\label{eq_inclusion1}
V^{\gamma}B_{\bff}\subseteq W^{\gamma}B_{\bff}:=\varphi_0(V^{\gamma}\widetilde{B}_g^{(0)})\quad\text{for all}\quad\gamma\in\Q.
\end{equation}
Note that by definition $W^{\bullet}B_{\bff}$ is an exhaustive, decreasing filtration indexed by rational numbers, which is discrete and left continuous
(since the microlocal filtration on $\widetilde{B}_g$ has these properties)
and Proposition~\ref{prop_properties_phi}i) implies that each $W^{\gamma}B_{\bff}$ is a $\cD_X$-submodule of $B_{\bff}$.
This filtration also satisfies
\begin{equation}\label{eq_propertyW1}
t_i\cdot W^{\gamma}B_{\bff}\subseteq W^{\gamma+1}B_{\bff}\quad\text{for all}\quad \gamma\in\Q, 1\leq i\leq d.
\end{equation}
Indeed, note that if $u\in V^{\gamma}\widetilde{B}_g^{(0)}$, then $-\partial_z^{-1}\partial_{y_i}u\in V^{\gamma+1}\widetilde{B}_g^{(0)}$ and
it follows from properties ii) and iv) in Proposition~\ref{prop_properties_phi} that
$$t_i\cdot \varphi_0(u)=\varphi_0(-\partial_z^{-1}\partial_{y_i}u)\in W^{\gamma+1}B_{\bff}.$$
We also have
\begin{equation}\label{eq_propertyW3}
\partial_{t_i}\cdot W^{\gamma}B_{\bff}\subseteq W^{\gamma-1}B_{\bff}\quad\text{for all}\quad \gamma\in\Q, 1\leq i\leq d.
\end{equation}
Indeed, if $u\in V^{\gamma}\widetilde{B}_g^{(0)}$, then $\partial_zy_iu\in V^{\gamma-1}\widetilde{B}_g^{(0)}$,
and it follows from properties ii) and iii) in Proposition~\ref{prop_properties_phi} that
$$\partial_{t_i}\cdot \varphi_0(u)=\varphi_0(\partial_zy_iu)\in W^{\gamma-1}B_{\bff}.$$
In particular, we see that each $W^{\gamma}B_{\bff}$ is a $V^0{\mathcal R}$-submodule of $B_{\bff}$. 

Furthermore, for every $\gamma\in\Q$, we have
\begin{equation}\label{eq_propertyW2}
s+\gamma\quad\text{is nilpotent on}\quad {\rm Gr}_W^{\gamma}(B_{\bff}).
\end{equation}
Indeed, assertion v) in Proposition~\ref{prop_properties_phi} gives $\varphi(su)=s\varphi(u)$ for every $u\in \widetilde{B}_g^{(0)}$ and 
we know that $s+\gamma$ is nilpotent on ${\rm Gr}^{\gamma}_V(\widetilde{B}_g)$. 

We can now prove the inclusion (\ref{eq_inclusion1}). 
If $\gamma$, $\gamma'$ are distinct rational numbers, then both $s+\gamma$ and $s+\gamma'$ are nilpotent on
\begin{equation}\label{biquotient}
\frac{V^{\gamma}B_{\bff}\cap W^{\gamma'}B_{\bff}}{(V^{>\gamma}B_{\bff}\cap W^{\gamma'}B_{\bff})+(V^{\gamma}B_{\bff}\cap W^{>\gamma'}B_{\bff})}
\end{equation}
(this follows from (\ref{eq_propertyW2}) and the fact that $s+\gamma$ is nilpotent on ${\rm Gr}_V^{\gamma}(B_{\bff})$ by definition of the $V$-filtration on $B_{\bff}$). 
This implies that the quotient in (\ref{biquotient}) is $0$. We deduce that
\begin{equation}\label{eq_incl_for_gamma}
V^{\gamma}B_{\bff}\subseteq W^{\gamma}B_{\bff}+V^{>\gamma}B_{\bff}\quad\text{for all}\quad\gamma\in\Q.
\end{equation}
Indeed, if $u\in V^{\gamma}B_{\bff}$, since $W^{\bullet}B_{\bff}$ is exhaustive, there is $\gamma'$ such that $u\in W^{\gamma'}B_{\bff}$. 
If $\gamma'\geq \gamma$, then we are done. Suppose now that $\gamma'<\gamma$. The fact that the quotient in (\ref{biquotient}) is $0$
implies that we can write $u=u_1+u_2$, with $u_1\in V^{>\gamma}B_{\bff}\cap W^{\gamma'}B_{\bff}$ and $u_2\in V^{\gamma}B_{\bff}\cap W^{>\gamma'}B_{\bff}$. 
Note that $u$ lies in the right-hand side of (\ref{eq_incl_for_gamma}) if and only if $u_2$ does. Also, we have $u_2\in V^{\gamma}B_{\bff}\cap W^{\gamma''}B_{\bff}$
for some $\gamma''>\gamma'$. We can repeat the argument with $u$ replaced by $u_2$; since $W^{\bullet}B_{\bff}$ is discrete, we see that after finitely many steps we conclude that $u\in W^{\gamma}B_{\bff}+V^{>\gamma}B_{\bff}$. 

Using the fact that the filtration $V^{\bullet}B_{\bff}$ is discrete, we deduce from (\ref{eq_incl_for_gamma}) that for every $\gamma$, $\gamma'\in\Q$, we have
\begin{equation}\label{eq_incl_for_gamma2}
V^{\gamma}B_{\bff}\subseteq W^{\gamma}B_{\bff}+V^{\gamma'}B_{\bff}.
\end{equation}
We next note that given $\gamma\in\Q$, it follows from property iii) in the definition of the $V$-filtration on $B_{\bff}$ that 
there is an integer $q_0$ such that for every integer $q\geq q_0$, we have
$$V^{\gamma+q}B_{\bff}\subseteq V^{q-q_0}{\mathcal R}\cdot V^{\gamma+q_0}B_{\bff}.$$
On the other hand, since $V^{\gamma+q_0}B_{\bff}$ is a finitely generated $V^0{\mathcal R}$-module and $W^{\bullet}B_{\bff}$ is exhaustive, there is $\beta$ 
such that
$V^{\gamma+q_0}B_{\bff}\subseteq W^{\beta}B_{\bff}$. By taking $q$ such that $q-q_0+\beta\geq\gamma$, we conclude that
$$V^{\gamma}B_{\bff}\subseteq W^{\gamma}B_{\bff}+V^{\gamma+q}B_{\bff}\subseteq W^{\gamma}B_{\bff}+V^{q-q_0}{\mathcal R}\cdot W^{\beta}B_{\bff}
\subseteq W^{\gamma}B_{\bff}+W^{\beta+q-q_0}B_{\bff}\subseteq W^{\gamma}B_{\bff},$$
where the first inclusion follows from (\ref{eq_incl_for_gamma2}) and the third one follows from (\ref{eq_propertyW1}) and (\ref{eq_propertyW3}).
This completes the proof of (\ref{eq_inclusion1}).

In order to complete the proof of the theorem, it is enough to also show that 
$$\varphi_0(V^{\gamma}\widetilde{B}_g^{(0)})\subseteq V^{\gamma}B_{\bff}\quad\text{for all}\quad \gamma\in\Q.$$
In fact, we will prove the equivalent statement that
\begin{equation}\label{eq_inclu_U}
V^{\gamma}\widetilde{B}_g\subseteq U^{\gamma}\widetilde{B}_g:=\bigoplus_{m\in\Z}\partial_z^{-m}\varphi_0^{-1}(V^{\gamma-m}B_{\bff})
\quad\text{for all}\quad \gamma\in\Q.
\end{equation}

It is clear from the definition that $U^{\bullet}\widetilde{B}_g$ is an exhaustive, decreasing filtration indexed by rational numbers (since the $V$-filtration
on $B_{\bff}$ has these properties). Moreover, if $\ell$ is a positive integer such that $V^{\gamma}B_{\bff}$ is constant for $\gamma$ in each interval of the form 
$\left(\tfrac{i-1}{\ell},\tfrac{i}{\ell}\right]$, with $i\in\Z$, then it follows from the definition that $U^{\gamma}\widetilde{B}_{g}$ is constant for $\gamma$ in such an interval.
Therefore $U^{\bullet}\widetilde{B}_g$ is discrete and left continuous.

Note next that by Proposition~\ref{prop_properties_phi}i), every $U^{\gamma}\widetilde{B}_g$ is a $\cD_X$-submodule of $\widetilde{B}_g$. 
Moreover, it follows directly from the definition that
$$\partial_z^j\cdot U^{\gamma}\widetilde{B}_g\subseteq U^{\gamma-j}\widetilde{B}_g\quad\text{for all}\quad \gamma\in\Q, j\in\Z.$$
We also have
\begin{equation}\label{eq_mult_by_z}
z\cdot U^{\gamma}\widetilde{B}_g\subseteq U^{\gamma+1}\widetilde{B}_g\quad\text{for all}\quad \gamma\in\Q.
\end{equation}
Indeed, if $u\in\varphi_0^{-1}(V^{\gamma-m}B_{\bff})$ for some $m\in\Z$, then using the fact that $[z,\partial_z^{-m}]=m\partial_z^{-m-1}$
(see Lemma~\ref{lem0_appendix} in the Appendix), we have
$$z\partial_z^{-m}u=\partial_z^{-m}(zu+m\partial_z^{-1}u)=\partial_z^{-m-1}(\partial_zzu+mu)=\partial_z^{-m-1}(-su+mu).$$
Note that $-su+mu\in \widetilde{B}_g^{(0)}$ 
and using Proposition~\ref{prop_properties_phi} we see that
$$\varphi_0(-su+mu)=-(s-m)\varphi_0(u)\in V^{\gamma-m}B_{\bff},$$
hence $z\partial_z^{-m}u\in U^{\gamma+1}\widetilde{B}_g$, proving (\ref{eq_mult_by_z}).
In particular, we see that each $U^{\gamma}\widetilde{B}_g$ is a $V^0\widetilde{\mathcal R}$-module. 

Finally, $s+\gamma$ is nilpotent on ${\rm Gr}_U^{\gamma}(\widetilde{B}_g)$ for every $\gamma\in \Q$. 
Indeed, suppose that $u_0\in \varphi_0^{-1}(V^{\gamma-m}B_{\bff})$. In this case 
$$\varphi_0\big((s+\gamma-m)^Nu\big)=(s+\gamma-m)^N\varphi_0(u)\in V^{>\gamma-m}B_{\bff}$$
for $N\gg 0$, where the equality follows from Proposition~\ref{prop_properties_phi}v). 
Since $P(s)\partial_z^{-m}=\partial_z^{-m}P(s-m)$ for every $P\in\C[s]$ (see Lemma~\ref{lem2_appendix} in Appendix), it follows that
$$(s+\gamma)^N\partial_z^{-m}u=\partial_z^{-m}(s+\gamma-m)^Nu\in U^{>\gamma}\widetilde{B}_g\quad\text{for}\quad N\gg 0.$$

We can now prove the inclusion (\ref{eq_inclu_U}). Since the argument is very similar to that we used in the proof of (\ref{eq_inclusion1}),
we omit some of the details. First, we see that 
\begin{equation}\label{eq_extra}
\frac{V^{\gamma}\widetilde{B}_g\cap U^{\gamma'}\widetilde{B}_g}{(V^{>\gamma}\widetilde{B}_g\cap U^{\gamma'}\widetilde{B}_g)+
(V^{\gamma}\widetilde{B}_g+U^{>\gamma'}\widetilde{B}_g)}=0\quad\text{for}\quad\gamma\neq\gamma'
\end{equation}
using the fact that both $s+\gamma$ and $s+\gamma'$ are nilpotent on this quotient. As before, we use the fact that $U^{\bullet}\widetilde{B}_g$ is exhaustive
and discrete and $V^{\bullet}\widetilde{B}_g$ is discrete
to deduce from (\ref{eq_extra}) that for every $\gamma,\gamma'\in\Q$, we have
\begin{equation}\label{eq_inclu_inU1}
V^{\gamma}\widetilde{B}_g\subseteq U^{\gamma}\widetilde{B}_g+V^{\gamma'}\widetilde{B}_g.
\end{equation}
Let us fix now $\gamma\in\Q$. 
Since $V^0\widetilde{B}_g$ is a finitely generated $V^0\widetilde{\mathcal R}$-module and each $U^{\beta}\widetilde{B}_g$ is a  $V^0\widetilde{\mathcal R}$-module,
it follows that there is $\beta\in\Q$ such that $V^0\widetilde{B}_g\subseteq U^{\beta}\widetilde{B}_g$. 
If we take $\gamma'\in\Z$ such that $\gamma'+\beta\geq\gamma$, then using (\ref{eq_inclu_inU1}) and (\ref{eq_mult_partial}) we conclude that
$$V^{\gamma}\widetilde{B}_g\subseteq U^{\gamma}\widetilde{B}_g+V^{\gamma'}\widetilde{B}_g=
U^{\gamma}\widetilde{B}_g+\partial_z^{-\gamma'}\cdot V^0\widetilde{B}_g$$
$$\subseteq 
U^{\gamma}\widetilde{B}_g+\partial_z^{-\gamma'}\cdot U^{\beta}\widetilde{B}_g
\subseteq U^{\gamma}\widetilde{B}_g+U^{\gamma'+\beta}\widetilde{B}_g\subseteq U^{\gamma}\widetilde{B}_g.$$
This completes the proof of the theorem.
\end{proof}

We end this section with another application of the map $\varphi$, relating $b$-functions (with respect to $f_1,\ldots,f_d$) to microlocal $b$-functions (with respect to $g$). 

\begin{prop}\label{prop_b_fcn}
For every $m\in\Z$ and every $u\in \widetilde{B}_g^{(m)}$, we have
$$\widetilde{b}_u(s-m)=b_{\varphi(u)}(s).$$
\end{prop}

\begin{proof}
By definition of $\widetilde{b}_u(s)$, working locally on $X$, we can find $P\in V^1\widetilde{\mathcal R}=\sum_{a-b\geq 1}\cD_Yz^a\partial_z^b$ such that
\begin{equation}\label{eq1_prop_b_fcn}
\widetilde{b}_u(s)u=Pu.
\end{equation}
 We may and will assume that ${\rm deg}(P)=0$. This implies that we can write
$P=\sum_{i=1}^d\partial_{y_i}P_i$ for some $P_1,\ldots,P_d$ of degree $1$. If we put $Q_i=\partial_zP_i$ for all $i$, then
$P=\sum_{i=1}^d\partial_{y_i}\partial_z^{-1}Q_i$. 
Applying $\varphi$ to (\ref{eq1_prop_b_fcn}), we obtain
\begin{equation}\label{eq2_prop_b_fcn}
\varphi\big(\widetilde{b}_u(s)u)=\varphi(Pu).
\end{equation}
By Proposition~\ref{prop_properties_phi}v), the left-hand side of (\ref{eq2_prop_b_fcn}) is equal to $\widetilde{b}_u(s-m)\varphi(u)$. 
On the other hand, it follows from properties ii) and iv) in Proposition~\ref{prop_properties_phi}  that the right-hand side of (\ref{eq2_prop_b_fcn}) is equal to
$$\sum_{i=1}^d\varphi(\partial_{y_i}\partial_{z}^{-1}Q_iu)=-\sum_{i=1}^dt_i\cdot\varphi(Q_iu).$$
Note that for every $i$, we have
$$Q_i\in\sum_{\alpha,\beta,a,b}\cD_Xy^{\alpha}\partial_y^{\beta}z^a\partial_z^b,$$
where $a,b\in\Z$ are such that $a\geq b$ and $a\geq 0$, while $\alpha,\beta\in\Z^d$ are such that $|\alpha|\leq |\beta|$ (this follows from the condition on $a$ and $b$
and the fact that ${\rm deg}(Q_i)=0$). We can write $z^a\partial_z^b=(z^a\partial_z^a)\partial_z^{b-a}\in \C[s]\cdot\partial_z^{b-a}$
(see Lemma~\ref{lem2_appendix}
 in Appendix), hence using Proposition~\ref{prop_properties_phi} we conclude that for every $i$, we have
 $$\varphi(Q_iu)\in \sum_{|\alpha|\leq |\beta|}\cD_X[s]\cdot \partial_t^{\alpha}t^{\beta}\cdot \varphi(u)\subseteq V^0{\mathcal R}\cdot \varphi(u).$$
 We thus conclude that
 $$\widetilde{b}_u(s-m)\varphi(u)\in \sum_{i=1}^dt_i\cdot V^0{\mathcal R}\cdot \varphi(u)\subseteq V^1{\mathcal R}\cdot \varphi(u),$$
 hence by the definition of $b_{\varphi(u)}(s)$, we have
 \begin{equation}\label{eq_div1}
 b_{\varphi(u)}(s)\quad\text{divides}\quad \widetilde{b}_u(s-m).
 \end{equation}
 
 Going in the opposite direction, it follows from the definition of $b_{\varphi(u)}(s)$ that, working locally on $X$, there is $T\in V^1{\mathcal R}$
 such that 
 $$b_{\varphi(u)}(s)\varphi(u)=T\varphi(u).$$
 We can write
 $$T =\sum_{|\alpha|\geq |\beta|+1}T_{\alpha,\beta}t^{\alpha}\partial_t^{\beta},\quad\text{with}\quad T_{\alpha,\beta}\in\cD_X.$$
 Let us consider
 $$\widetilde{T}=\sum_{|\alpha|\geq |\beta|+1}T_{\alpha,\beta}(-y)^{\alpha}\partial_y^{\beta}\partial_z^{|\beta|-|\alpha|}\in V^1\widetilde{\mathcal R},$$
 of degree $0$, so $\widetilde{T}u\in \widetilde{B}_g^{(m)}$. 
 Using Proposition~\ref{prop_properties_phi}, we see that
 $$\varphi\big(b_{\varphi(u)}(s+m)u\big)=b_{\varphi(u)}(s)\varphi(u)=T\varphi(u)=\varphi(\widetilde{T} u).$$
 Since the restriction of $\varphi$ to $\widetilde{B}_g^{(m)}$ is injective, we conclude that
 $$b_{\varphi(u)}(s+m)u=\widetilde{T} u.$$
 We deduce using the definition of $\widetilde{b}_u$ that
 \begin{equation}\label{eq_div2}
 \widetilde{b}_u(s)\quad\text{divides}\quad b_{\varphi(u)}(s+m).
 \end{equation}
 By combining (\ref{eq_div1}) and (\ref{eq_div2}), we see that $\widetilde{b}_u(s-m)$ and $b_{\varphi(u)}(s)$ are monic polynomials that divide each other, hence they are equal. 
\end{proof} 

\begin{rmk}\label{rmk_relation_b_function}
Note that if we apply the above proposition for $u=\delta_g\in \widetilde{B}_g^{(0)}$, then we recover the fact that the Bernstein-Sato polynomial of the ideal 
$(f_1,\ldots,f_d)$ coincides with the microlocal $b$-function of $\delta_g$ (which, as we have mentioned in Section~\ref{section_review}, is equal to $b_g(s)/(s+1)$).
This is the main result in \cite{Mustata}. 
\end{rmk}

\section{The minimal exponent of a local complete intersection subscheme}

Our goal in this section is to define and study the minimal exponent of a local complete intersection subscheme. Let $X$ be a smooth, irreducible complex algebraic variety
and $Z$ a (nonempty) proper closed subscheme of $X$.

\begin{rmk}\label{lct_bounded_codim}
We note that in general we have $\lct(X, Z)\leq {\rm codim}_X(Z)$. Indeed, the inclusion $Z_{\rm red}\hookrightarrow Z$ implies $\lct(X, Z)\leq \lct(X, Z_{\rm red})$, hence we may assume 
that $Z$ is reduced. If $U$ is an open subset of $X$ such that $U\cap Z$ is smooth and irreducible of codimension $r$ in $U$, then
$\lct(X, Z)\leq \lct(U, Z\cap U)=r$. 
\end{rmk}

\begin{rmk}\label{rel_with_lct2}
Note also that if $Z$ is Cohen-Macaulay, of pure codimension $r$, and
$\lct(X,Z)=r$, then $Z$ is reduced. Indeed, if this is not the case, then $Z$ is not generically reduced (being Cohen-Macaulay). 
It follows that we have an irreducible component $Z_0$ of $Z$ such that the local ring $\cO_{Z,Z_0}$ is not a field; therefore the embedding dimension $m$ of
$\cO_{Z,Z_0}$ is positive. After possibly replacing $X$ by a suitable open subset that intersects $Z_0$ nontrivially, we may assume that 
$Z$ is irreducible, $Z_0$ is smooth, and
there is a smooth, irreducible 
subvariety $W$ of $X$ of dimension $\dim(Z_0)+m=n-r+m$ such that $Z$ is contained in $W$ and, in fact, 
the ideal defining $Z$ in $W$ is contained in the ideal
$I_{Z_0/W}^2$, where $I_{Z_0/W}$ is the ideal defining $Z_0$ in $W$. By considering the exceptional divisor on the blow-up of $W$ along $Z_0$ and the description of $\lct(W,Z)$ in terms of
log resolutions (see \cite[Example~9.3.16]{Lazarsfeld}), it follows easily that $\lct(W,Z)\leq{\rm codim}_W(Z_0)/2=m/2$.
On the other hand, we have $\lct(X,Z)=\lct(W,Z)+{\rm codim}_X(W)$ (see for example \cite[Proposition~2.6]{Mustata0}). Therefore we have
$$\lct(X,Z)=r-m+\tfrac{m}{2}=r-\tfrac{m}{2}\leq r-\tfrac{1}{2}.$$
\end{rmk}

Suppose now that $Z$ is a local complete intersection, of pure codimension $r\geq 1$ in $X$. 
We first consider the case when $Z$ is globally a complete intersection, that is, there are $f_1,\ldots,f_r\in\cO_X(X)$ such that $Z$ is defined by the ideal
generated by $f_1,\ldots,f_r$. In this case we consider the $V$-filtration on $B_{\bff}$. Note that by (\ref{formula_lct}) and Remark~\ref{lct_bounded_codim},
we have $\delta_{\bff}\not\in V^{\gamma}B_{\bff}$ for $\gamma>r$. 

\begin{defi}
We define the minimal exponent $\widetilde{\alpha}(Z)$, as explained in the introduction, by the formula
\begin{equation}\label{eq3_intro_v2}
\widetilde{\alpha}(Z)
= \left\{
\begin{array}{cl}
\sup\{\gamma>0\mid \delta_{\mathbf f}\in V^{\gamma}B_{\bff}\} , & \text{if}\,\,\delta_{\mathbf f}\not\in V^rB_{\mathbf f}; \\[2mm]
\sup\{r-1+q+\gamma\mid F_qB_{\bff}\subseteq V^{r-1+\gamma}B_{\bff}\} , & \text{if}\,\,\delta_{\mathbf f}\in V^rB_{\mathbf f},
\end{array}\right.
\end{equation}
where in the latter case, the supremum is over all nonnegative integers $q$ and all rational numbers $\gamma\in (0,1]$ with the property that
$\partial_t^{\beta}\delta_{\mathbf f}\in V^{r-1+\gamma}B_{\bff}$ for all $\beta=(\beta_1,\ldots,\beta_r)\in\Z_{\geq 0}^r$, with $\beta_1+\ldots+\beta_r\leq q$. 
We note that the value of $\widetilde{\alpha}(Z)$ does not depend just on $Z$, but also on $X$ (for the precise way in which it depends on $X$, see
Proposition~\ref{dep_on_X} below). Because of this, whenever the ambient variety is not clear from the context, we write $\widetilde{\alpha}(X,Z)$
instead of $\widetilde{\alpha}(Z)$. 
\end{defi}

\begin{rmk}\label{max_vs_sup}
Since the $V$-filtration is left continuous, the supremum in the definition is a maximum, unless $\widetilde{\alpha}(Z)=\infty$ (which happens if and only if 
$Z$ is smooth, see Remark~\ref{rmk_singular} below). 
\end{rmk}

\begin{rmk}\label{rmk_min_exp_lct}
It follows from the definition and (\ref{formula_lct}) that
\begin{equation}\label{eq_rmk_min_exp_lct}
\lct(X,Z)=\min\big\{\widetilde{\alpha}(Z),r\big\}.
\end{equation}
\end{rmk}

\begin{rmk}\label{rmk_defi}
If $q_1$ and $q_2$ are nonnegative integers and $\gamma_1,\gamma_2\in (0,1]$ are rational numbers such that $q_1+\gamma_1\geq q_2+\gamma_2$,
and if $F_{q_1}B_{\bff}\subseteq V^{r-1+\gamma_1}B_{\bff}$, then 
$F_{q_2}B_{\bff}\subseteq V^{r-1+\gamma_2}B_{\bff}$. Indeed, this is clear if $q_1=q_2$, and if this is not the case,
then our hypothesis implies $q_2=q_1-1$ and it is enough to show that if $u=\partial_t^{\beta}\delta_{\bff}$, with $|\beta|\leq q_1-1$, then $u\in V^rB_{\bff}$. 
The assumption implies $\partial_{t_i}u\in V^{r-1}B_{\bff}$ for all $i$ and thus $t_i\partial_{t_i}u\in V^{r}B_{\bff}$. We conclude that 
$$(s+r)u=(-\partial_{t_1}t_1-\ldots-\partial_{t_r}t_r+r)u=-\sum_{i=1}^rt_i\partial_{t_i}u\in V^rB_{\bff}.$$
For every $\gamma\neq r$, since $s+\gamma$ is nilpotent on ${\rm Gr}_V^{\gamma}(B_{\bff})$, it follows that $s+r$ is invertible on this graded piece.
Since $(s+r)u\in V^rB_{\bff}$, using the discreteness of the $V$-filtration, we conclude that $u\in V^rB_{\bff}$.

The same argument shows that in order to have $F_qB_{\bff}\subseteq V^{r-1+\gamma}B_{\bff}$, it is enough to require $\partial_t^{\beta}\delta_{\bff}\in V^{r-1+\gamma}B_{\bff}$ 
for all $\beta$ with $|\beta|=q$. 
\end{rmk}

\begin{rmk}\label{union_open_sets}
Suppose that $U_1,\ldots,U_N$ are open subsets of $X$ such that all $Z\cap U_i$ are nonempty and
$Z\subseteq U_1\cup\ldots\cup U_N$. Since $\partial_t^{\beta}\delta_{\bff}\in V^{\gamma}B_{\bff}$
if and only if the same containment holds on each $U_i$ (note that the condition automatically holds over $X\smallsetminus Z$ by Remark~\ref{restr_complement_Z}),
it follows using also the assertion in Remark~\ref{rmk_defi} that
$$\widetilde{\alpha}(X, Z)=\min_{1\leq i\leq N}\widetilde{\alpha}(U_i, Z\cap U_i).$$
\end{rmk}

\begin{rmk}\label{indep_defi}
The definition of $\widetilde{\alpha}(Z)$ does not depend on the choice of $f_1,\ldots,f_r$. By 
taking an affine open cover of $X$ and using
Remark~\ref{union_open_sets}, we see that it is enough to prove this assertion when $X$ is affine. 
Suppose now that we have regular functions $f_1,\ldots,f_r$ and $g_1,\ldots,g_r$ such that $(f_1,\ldots,f_r)=(g_1,\ldots,g_r)$. 
The condition $\delta_{\bff}\in V^{\gamma}B_{\bff}$ is equivalent to $\lct(X,Z)\geq \gamma$, hence it is independent of
the choice of generators for the ideal. We thus only need to show that if $q\in\Z_{\geq 0}$ and $\gamma\in (0,1]$ is a rational number, then
$F_qB_{\bff}\subseteq V^{r-1+\gamma}B_{\bff}$ if and only if $F_qB_{\mathbf g}\subseteq V^{r-1+\gamma}B_{\mathbf g}$.

Let us write $g_i=\sum_ja_{i,j}f_j$ for $1\leq i\leq r$. Note that $D={\rm det}(a_{i,j})$ does not vanish at any point in $Z$. After replacing $X$ by the complement 
of the zero-locus of $D$, we may assume that  $D$ is invertible (see 
Remark~\ref{union_open_sets}). In this case
$$u\colon X\times\A^r\to X\times \A^r,\quad u(x,t_1,\ldots,t_r)=\left(x,\sum_ja_{1,j}t_j,\ldots,\sum_ja_{r,j}t_j\right)$$
is an isomorphism such that $u(X\times\{0\})=X\times\{0\}$ and $B_{\mathbf g}=u_+B_{\bff}$.
We thus have an isomorphism $u^*$ of ${\mathcal R}=\cD_X\langle t_1,\ldots,t_r,\partial_{t_1},\ldots,\partial_{t_r}\rangle$
that keeps $\cD_X$ fixed and maps each $t_i$ to a linear form in $t_1,\ldots,t_r$ and an isomorphism of ${\mathcal R}$-modules
$\tau\colon B_{\mathbf g}\to B_{\bff}$ (where we view $B_{\bff}$ as an ${\mathcal R}$-module via $u^*$). This clearly has the property
that $\tau(F_qB_{\mathbf g})=F_qB_{\bff}$ for every $p$ and using the uniqueness of the $V$-filtration, we see that 
$\tau(V^{\gamma}B_{\mathbf g})=V^{\gamma}B_{\bff}$ for all $\gamma\in\Q$. It is then clear that we have 
$F_qB_{\bff}\subseteq V^{r-1+\gamma}B_{\bff}$ if and only if $F_qB_{\mathbf g}\subseteq V^{r-1+\gamma}B_{\mathbf g}$. 
\end{rmk}

Suppose now that $Z$ is an arbitrary local complete intersection closed subscheme of $X$, of pure codimension $r\geq 1$. We can find open subsets
$U_1,\ldots,U_N$ of $X$ with $Z\subseteq\bigcup_{i=1}^NU_i$ such that each $Z\cap U_i$ is nonempty and defined in $U_i$ by an ideal generated by
$r$ regular functions on $U_i$. In particular, each $\widetilde{\alpha}(U_i, Z\cap U_i)$ is well-defined.

\begin{defi}
With the above notation, the \emph{minimal exponent} of $Z$ is
$$\widetilde{\alpha}(Z)=\widetilde{\alpha}(X,Z):=\min_{1\leq i\leq N}\widetilde{\alpha}(U_i, Z\cap U_i).$$
\end{defi}

\begin{rmk}\label{union_open_sets2}
It is easy to see, using Remark~\ref{union_open_sets}, that the definition is independent of the choice of open subsets $U_1,\ldots,U_N$. Moreover, given 
any open subsets $V_1,\ldots,V_m$ of $X$, with $Z\subseteq\bigcup_{j=1}^mV_j$ such that each $Z\cap V_j$ is nonempty, we have
$$\widetilde{\alpha}(X, Z)=\min_{1\leq j\leq m}\widetilde{\alpha}(V_j, Z\cap V_j).$$
\end{rmk}

\begin{rmk}
In the case of hypersurfaces (that is, $r=1$), we recover the usual definition of the minimal exponent by (\ref{eq_char_min_exp}).
\end{rmk}

\begin{prop}\label{prop_smooth_morphism}
If $\pi\colon Y\to X$ is a surjective smooth morphism of smooth, irreducible varieties, and $Z$ is a local complete intersection closed subscheme of $X$, of pure codimension $r$,
then $\widetilde{\alpha}(X,Z)=\widetilde{\alpha}\big(Y,\pi^{-1}(Z)\big)$.
\end{prop}

\begin{proof}
We may and will assume that $Z$ is defined in $X$ by the ideal generated by $f_1,\ldots,f_r\in\cO_X(X)$ and let $g_i=f_i\circ\pi$ for $1\leq i\leq r$. 
Using the fact that $\pi$ is smooth, it is then straightforward to see that we have an isomorphism
$$B_{\mathbf g}\simeq \pi^*B_{\bff}$$
such that for every $p\in\Z_{\geq 0}$ and every $\alpha\in\Q$, we get
$$F_pB_{\mathbf g}\simeq \pi^*F_pB_{\bff}\quad\text{and}\quad V^{\alpha}B_{\mathbf g}\simeq \pi^*V^{\alpha}B_{\bff}.$$
The assertion in the proposition then follows directly from the definition of the minimal exponent. 
\end{proof}

Our next goal is to describe the minimal exponent of $Z$ via the minimal exponent of a hypersurface. 
Suppose that $Z$ is a nonempty closed subscheme of $X$, of pure codimension $r\geq 1$,
whose ideal is generated by $f_1,\ldots,f_r\in\cO_X(X)$. We put $g=\sum_{i=1}^rf_iy_i\in\cO_X(X)[y_1,\ldots,y_r]$. Let
$U=X\times \big(\A^r\smallsetminus\{0\}\big)\subseteq Y=X\times\A^r$. We will freely use the notation in Section~\ref{section_rel_Vfilt}. 
The following is the key observation:

\begin{lem}\label{lem_obs}
If $\gamma\in\Q$ and $\alpha\in\Z_{\geq_0}^r$ are such that $y^{\alpha}\partial_z^{|\alpha|}\delta_g\in V^{\gamma}\widetilde{B}_g\smallsetminus V^{>\gamma}\widetilde{B}_g$
and $y^{\alpha}\partial_z^{|\alpha|}\delta_g\vert_U\in V^{>\gamma}\widetilde{B}_g\vert_U$, then $\gamma\geq r$ and $\gamma\in\Z$. 
\end{lem}

\begin{proof}
By assumption, if $u$ is the class of $y^{\alpha}\partial_z^{|\alpha|}\delta_g$ in ${\rm Gr}_V^{\gamma}(\widetilde{B}_g)$, then $u\neq 0$, but there is $N$ such that
$(y_1,\ldots,y_r)^Nu=0$. This implies that there is $\beta\in\Z_{\geq 0}$ such that $v=y^{\beta}u\neq 0$, but $(y_1,\ldots,y_r)v=0$. Note that 
$u\in {\rm Gr}_V^{\gamma}(\widetilde{B}^{(0)}_g)$, hence $v\in {\rm Gr}_V^{\gamma}(\widetilde{B}^{(m)}_g)$, where $m=|\beta|\geq 0$. 
By Lemma~\ref{lem_theta}, we have $(\theta_y-s)v=mv$.

On the other hand, since $y_iv=0$ for all $i$, we have 
$$\theta_yv=\sum_{i=1}^ry_i\partial_{y_i}u_i=-rv+\sum_{i=1}^r\partial_{y_i}y_iv=-rv,$$
hence $(s+m+r)v=\theta_yv+rv=0$. By definition of the $V$-filtration, $s+\gamma$ is nilpotent on ${\rm Gr}^{\gamma}_V(\widetilde{B}_g)$,
and thus $s+\lambda$ is invertible on ${\rm Gr}^{\gamma}_V(\widetilde{B}_g)$ for every $\lambda\neq\gamma$. Since $v\neq 0$,
we conclude that $\gamma=m+r\geq r$, which completes the proof.
\end{proof}

We can prove now the equality $\widetilde{\alpha}(Z)=\widetilde{\alpha}(g\vert_U)$.

\begin{proof}[Proof of Theorem~\ref{thm2_intro}]
It is enough to show that for every $\beta\in\Q_{>0}$, we have $\widetilde{\alpha}(Z)\geq\beta$ if and only if $\widetilde{\alpha}(g\vert_U)\geq \beta$. 
We treat separately the cases when $\beta\leq r$ and when $\beta>r$.

If $\beta\leq r$, then by definition we have $\widetilde{\alpha}(Z)\geq\beta$ if and only if $\delta_{\bff}\in V^{\beta}B_{\bff}$. By Theorem~\ref{thm_V_filtration},
this is equivalent to $\delta_g\in V^{\beta}\widetilde{B}_g$. On the other hand, it follows from (\ref{eq_MLCT}) that $\widetilde{\alpha}(g\vert_U)\geq \beta$
if and only if $\delta_g\in V^{\beta}\widetilde{B}_g$ on $U$. It is clear that if $\delta_g\in V^{\beta}\widetilde{B}_g$ then this also holds after restricting to $U$
and we need to show that the converse holds. Arguing by contradiction,  let us assume that $\delta_g\in V^{\beta}\widetilde{B}_g$ on $U$, but
$\delta_g\not\in V^{\beta}\widetilde{B}_g$. In this case, let
$\beta'=\max\{\gamma\in\Q_{\geq 0}\mid \delta_g\in V^{\gamma}\widetilde{B}_g\}<\beta\leq r$. By 
assumption, we have $\delta_g\in V^{>\beta'}\widetilde{B}_g$ on $U$, hence 
Lemma~\ref{lem_obs} implies $\beta'\geq r$, a contradiction.

If $\beta>r$, let us write $\beta=r-1+q+\gamma$, where $q$ is a positive integer and $\gamma\in (0,1]$ is a rational number.
By definition, we have $\widetilde{\alpha}(Z)\geq\beta$ if and only if $\partial_t^{\alpha}\delta_{\bff}\in V^{r-1+\gamma}B_{\bff}$ for every $\alpha\in\Z_{\geq 0}^r$
with $|\alpha|\leq q$. By Theorem~\ref{thm_V_filtration}, this holds if and only if $y^{\alpha}\partial_z^{|\alpha|}\delta_g\in V^{r-1+\gamma}\widetilde{B}_g$
for all such $\alpha$. On the other hand, it follows from (\ref{eq_MLCT}) that $\widetilde{\alpha}(g\vert_U)\geq \beta$ if and only if 
$\delta_g\in V^{\beta}\widetilde{B}_g$ on $U$. 

Note that $U=U_1\cup\ldots\cup U_r$, where $U_i$ is the complement of the zero-locus of $y_i$. We thus see that if $y_i^q\partial_z^{q}\delta_g\in V^{r-1+\gamma}
\widetilde{B}_g$, we have $\partial_z^{q}\delta_g\in V^{r-1+\gamma}
\widetilde{B}_g$ on $U_i$, and thus $\delta_g\in V^{\beta}\widetilde{B}_g$ on $U_i$ by (\ref{eq_mult_partial}). 
We conclude that if $\widetilde{\alpha}(Z)\geq\beta$, then $\delta_g\in V^{\beta}\widetilde{B}_g$ on $U$ and thus $\widetilde{\alpha}(g\vert_U)\geq \beta$.

In order to prove the converse, we argue by contradiction: we assume that $\delta_g\in V^{\beta}\widetilde{B}_g$ on $U$, but
there is $\alpha\in\Z_{\geq 0}$ with $|\alpha|\leq q$ such that $y^{\alpha}\partial_z^{|\alpha|}\delta_g\not \in V^{r-1+\gamma}\widetilde{B}_g$.
Let 
$$\beta'=\max\{\eta\in\Q_{\geq 0}\mid y^{\alpha}\partial_z^{|\alpha|}\delta_g\in V^{\eta}\widetilde{B}_g\}.$$
Note that $\beta'<r-1+\gamma\leq r$. On the other hand, since $\delta_g\in V^{\beta}\widetilde{B}_g$ on $U$,
we have $y^{\alpha}\partial_z^{|\alpha|}\delta_g\in V^{\beta-|\alpha|}\widetilde{B}_g\subseteq V^{>\beta'}\widetilde{B}_g$ on $U$.
Applying Lemma~\ref{lem_obs}, we get $\beta'\geq r$, a contradiction.
\end{proof}

\begin{prop}\label{dep_on_X}
If $Z$ is a local complete intersection scheme of pure
dimension and $Z\hookrightarrow X$ is a closed embedding, where $X$ is a smooth, irreducible variety, then
$\widetilde{\alpha}(X,Z)-\dim(X)$ does not depend on $X$, but only on $Z$.
\end{prop}

\begin{proof}
Let us consider two embeddings $i\colon Z\hookrightarrow X$ and $i'\colon Z\hookrightarrow X'$, where both $X$ and $X'$ are smooth irreducible varieties.
After comparing both these embeddings with the diagonal embedding $(i,i')\colon Z\hookrightarrow X\times X'$, we see that we may assume that there is a smooth morphism
$p\colon X'\to X$ such that $p\circ i'=i$. Given any point $x\in Z$, we can choose regular systems of parameters $x_1,\ldots,x_n$ in $\cO_{X,i(x)}$ 
and $p^*(x_1),\ldots,p^*(x_n),y_1,\ldots,y_m$ in $\cO_{X',i'(x)}$ such that $y_1,\ldots,y_m$ vanish along $i'(Z)$.  In a suitable neighborhood of $i'(x)$, we get 
an \'{e}tale morphism $X'\to X\times \A^m$, given by $(p,y_1,\ldots,y_m)$ that maps $i'(Z)$ inside $X\times\{0\}$. After taking a suitable open cover of $Z$ and using the invariance of
the minimal exponent under \'{e}tale morphisms (see Proposition~\ref{prop_smooth_morphism}), we see that it is enough to prove that if
$X'=X\times\A^m$ and $i'=(i,0)$, then $\widetilde{\alpha}(X',Z)=\widetilde{\alpha}(X,Z)+m$. Of course, arguing by induction on $m$, we see that it is enough to treat the case $m=1$. 

We may and will assume that $X$ is affine, and the ideal defining $Z$ in $X$ is generated by a regular sequence $f_1,\ldots,f_r\in\cO_X(X)$. Of course, in this case 
the ideal defining $Z$ in $X'$ is $(f_1,\ldots,f_r,z)$, where $z$ denotes the coordinate on $\A^1$. Using the description of the minimal exponent in Theorem~\ref{thm2_intro},
we see that it is enough to show that if we put $U=X\times \big(\A^r\smallsetminus\{0\}\big)$ and $U'=X'\times \big(\A^{r+1}\smallsetminus\{0\}\big)$, 
and 
$$g=\sum_{i=1}^rf_iy_i\quad\text{and}\quad g'=zy_{r+1}+\sum_{i=1}^rf_iy_i,$$
then $\widetilde{\alpha}(g'\vert_{U'})=\widetilde{\alpha}(g\vert_U)$. Note that we can write $U'=U'_1\cup U'_2$, where $U'_2$ is given by $y_{r+1}\neq 0$
and $U'_1=U\times {\rm Spec}\big(\C[y_{r+1},z]$\big). Note that the hypersurface defined by $g'$ in $U'_2$ is smooth, while $g'\vert_{U'_1}=g\vert_{U_1}+zy_{r+1}$, hence
$$\widetilde{\alpha}(g'\vert_{U'})=\widetilde{\alpha}(g'\vert_{U'_1})=\widetilde{\alpha}(g\vert_{U_1})+1,$$
where the second equality follows from the Thom-Sebastiani theorem for minimal exponents (see \cite[Theorem~0.8]{Saito_microlocal}).
\end{proof}

We next use the description of the minimal exponent in Theorem~\ref{thm2_intro} to prove some basic properties of this invariant.
Until the end of this section, we assume that $X$ is a smooth, irreducible, $n$-dimensional variety and $Z$ is a closed subscheme of $X$ that
is locally a complete intersection, of pure codimension $r$. Recall that if $f\in\cO_X(X)$ is nonzero and $x\in X$, then the multiplicity
${\rm mult}_x(f)$ is the largest $d$ such that $f\in\frm_x^d$, where $\frm_x$ is the ideal defining $x$.
Before introducing a local version of the minimal exponent, we make the following

\begin{rmk}\label{rmk_singular}
We have $\widetilde{\alpha}(Z)<\infty$ if and only if $Z$ is singular (and in this case we have\footnote{For a sharper estimate, see Remark~\ref{rem_sharper} below.} 
$\widetilde{\alpha}(Z)\leq\tfrac{n+r}{2}$).  In order to see this, we may and will assume that $Z$ is defined by the ideal generated by
$f_1,\ldots,f_r\in\cO_X(X)$, and let $g=\sum_{i=1}^rf_iy_i$. We use the fact that by Theorem~\ref{thm2_intro}, we have
 $\widetilde{\alpha}(Z)=\widetilde{\alpha}(g\vert_U)$, where $U=X\times
\big(\A^r\smallsetminus\{0\}\big)$.
If $Z$ is smooth, we may assume that $f_1,\ldots,f_r$ are part of a system of coordinates $f_1,\ldots,f_n$ on $X$
(that is, $df_1,\ldots,df_n$ trivialize $\Omega_X$). In this case it is easy to see\footnote{For a more general statement, see Lemma~\ref{singloc} 
below.} that the singular locus of the hypersurface defined by $g$
is contained in $X\times\{0\}$ and thus $\widetilde{\alpha}(g\vert_U)=\infty$. Conversely, if $Z$ is singular, then we may and will assume that ${\rm mult}_x(f_1)\geq 2$
for some $x\in Z$. If $p=(1,0,\ldots,0)\in\A^r$, then $(x,p)\in U$ and ${\rm mult}_{(x,p)}(g)\geq 2$, hence $\widetilde{\alpha}(g\vert_U)\leq\tfrac{n+r}{2}$ by
\cite[Theorem~E(3)]{MP1}.
\end{rmk}

Suppose now that $x\in Z$ is a fixed point. We define the \emph{minimal exponent of} $Z$ \emph{at} $x$ as follows: it is clear from the definition that if $V'\subseteq V$
are open neighborhoods of $x$, then $\widetilde{\alpha}(V', Z\cap V')\geq \widetilde{\alpha}(V, Z\cap V)$. Moreover, there is $V$ such that for all $V'$ as above,
the inequality is an equality. Indeed, otherwise we have a decreasing sequence of open neighborhoods $V_i$ of $x$ such that $\big(\widetilde{\alpha}(V_i, Z\cap V_i)\big)_i$
is a strictly increasing sequence. If $\lim_{i\to\infty}\widetilde{\alpha}(V_i, Z\cap V_i)<\infty$, then we easily get a contradiction using the discreteness of the $V$-filtration. 
On the other hand, if $\lim_{i\to\infty}\widetilde{\alpha}(V_i, Z\cap V_i)=\infty$, then it follows from Remark~\ref{rmk_singular} that $x\in Z$ is a smooth point, hence it is enough to take
$V$ to be a neighborhood of $x$ such that $V\cap Z$ is smooth.

\begin{defi}
Given a point $x\in Z$, we put
$$\widetilde{\alpha}_x(Z):=\max_{V\ni x}\widetilde{\alpha}(V, Z\cap V),$$
where the maximum is over all open neighborhoods $V$ of $x$ in $X$.
Note that this maximum exists by the previous discussion. Moreover, it follows from Remark~\ref{rmk_singular} that $\widetilde{\alpha}_x(Z)=\infty$ if and only if $x$ is a smooth point of $Z$. As before, if the ambient space is not clear from the context, we write $\widetilde{\alpha}_x(X,Z)$ instead of $\widetilde{\alpha}_x(Z)$.
\end{defi}

\begin{rmk}
It is a consequence of the definition of the minimal exponent, of the discreteness of the $V$-filtration, and of Remark~\ref{rmk_singular}, that the set
$$\big\{\widetilde{\alpha}_x(Z)\mid x\in Z\}$$
is a finite set. It is then clear that 
$$\widetilde{\alpha}(Z)=\min\big\{\widetilde{\alpha}_x(Z)\mid x\in Z\big\}.$$
Furthermore, for every $\gamma\in\Q_{>0}$, the set $\big\{x\in Z\mid \widetilde{\alpha}_x(Z)\geq \gamma\big\}$
is open in $Z$. 
\end{rmk}

Slightly more generally, if $X$ is a smooth, but possibly disconnected variety, and $Z$ is a local complete intersection closed subscheme of $X$, with both $X$ and $Z$
pure dimensional,
then we can define $\widetilde{\alpha}_x(Z)$ for every $x\in Z$ by restricting to the connected component of $X$ that contains $x$. This is useful, for instance, in the setting of
Theorem~\ref{thm3_intro}, when we do not assume that the fibers of $\mu$ are connected.

Before we prove the main properties of the minimal exponent in general, we need to handle one such property in the special case of hypersurfaces.

\begin{lem}\label{lem_thm3_intro}
The assertion in Theorem~\ref{thm3_intro}ii) holds when $r=1$.
\end{lem}

We note that in the presence of a section $s\colon T\to X$, the openness assertion in the lemma is \cite[Theorem~E(2)]{MP1}. However, for our purpose it will be important 
to have the stronger openness assertion, that does not make reference to a section, since this is the one that will allow us to handle arbitrary codimension. The proof follows
closely the approach in \cite{MP1} (which in turn was modeled on the approach to prove the semicontinuity of log canonical thresholds via multiplier ideals, see
\cite[Example~9.5.41]{Lazarsfeld}). However, we include a detailed proof for the benefit of the reader.

The proof of assertion ${\rm ii_1)}$
makes use of the notion of \emph{Hodge ideals} for $\Q$-divisors, introduced and studied in \cite{MP3}. For every hypersurface $Z$ in a smooth variety $X$,
every nonnegative integer $p$, and every positive rational number $\alpha$, the corresponding Hodge ideal is denoted by $I_p(\alpha Z)$. For $p=0$, this is just the multiplier ideal
$\cJ\big((\alpha-\epsilon)Z\big)$, where $0 < \epsilon\ll 1$, see \cite[Proposition~9.1]{MP3} (since $Z$ is a hypersurface, we follow the traditional notation to write $\cJ(\lambda Z)$ for what we denoted before by $\cJ(\fra^{\lambda})$, where $\fra$ is the ideal defining $Z$). Moreover, it was shown in \cite{MP3} that many basic properties of multiplier ideals admit extensions to Hodge ideals. 

Recall that by definition of the log canonical threshold, we have
$\lct(X,Z)\geq \alpha$ if and only if $\cJ\big((\alpha-\epsilon)Z\big)=\cO_X$ for $0<\epsilon\ll 1$. Similarly, it was shown in \cite[Corollary~C]{MP1} that if $Z$ is reduced, $p$ is a nonnegative integer, and $\alpha\in\Q\cap (0,1]$, then $\widetilde{\alpha}(Z)\geq p+\alpha$ if and only if $I_p(\alpha Z)=\cO_X$. On the other hand, if $Z$ is not reduced, 
then we automatically have $\lct(X,Z)<1$, hence $\widetilde{\alpha}(Z)=\lct(X,Z)$ can be characterized using multiplier ideals.

\begin{proof}[Proof of Lemma~\ref{lem_thm3_intro}]
We first note that for every $t\in T$, the fiber $X_t$ is a smooth subvariety of the smooth variety $X$ that contains no component of the hypersurface $Z$. 
Locally around any $x\in X_t$, we can write $X_t$ as a transverse intersection of $\dim(T)$ smooth hypersurfaces in $X$, so that successively applying \cite[Theorem~E(1)]{MP1} to restrict to each
of these smooth hypersurfaces, we obtain
\begin{equation}\label{eq1_lem_thm3_intro}
\widetilde{\alpha}_x(X_t, Z_t)\leq \widetilde{\alpha}_x(X, Z).
\end{equation}
We next show that there is a nonempty open subset $T_0$ of $T$ such that for every $t\in T_0$ and every $x\in X_t$, the inequality in (\ref{eq1_lem_thm3_intro})
is an equality, thus proving the assertion in ${\rm ii_2)}$ in our setting. One way to see this is by using the characterization of the minimal exponent in terms of Hodge ideals and multiplier ideals and the fact that there is an open subset $T_0$ such that
$$\cJ(\lambda Z_t)=\cJ(\lambda Z)\cdot\cO_{Z_t}\quad\text{for all}\quad \lambda>0,\,\,t\in T_0$$
(see \cite[Theorem~9.5.35]{Lazarsfeld}) and, assuming that $Z$ is reduced and thus $Z_t$ is also reduced for general $t\in T$,
a similar formula holds for Hodge ideals
$$I_p(\lambda Z_t)=I_p(\lambda Z)\cdot\cO_{Z_t}\quad\text{for all}\quad p\in\Z_{\geq 0},\,\,\lambda\in\Q_{>0},\,\,t\in T_0$$
(see the last assertion in \cite[Theorem~13.1]{MP3}). Alternatively, one can use the characterization of the minimal exponent in terms of the $V$-filtration in
(\ref{eq_MLCT}) and the results concerning the behavior of the $V$-filtration with respect to non-characteristic restriction in \cite{DMST}.

We next prove the assertion in ${\rm ii_1)}$. For every $\alpha$, let
$$W_{\alpha}=\big\{x\in Z\mid \widetilde{\alpha}_x(X_{\mu(x)}, Z_{\mu(x)})\geq \alpha\big\}.$$

We first note that the assertion in ${\rm ii_1)}$ makes sense also when $T$ is not assumed to be a smooth variety, but just a
(reduced, but not necessarily irreducible) algebraic variety. However, in order to have the statement for such varieties of dimension $n$, it is enough to prove it
for smooth, irreducible, $n$-dimensional varieties. Indeed, using resolution of singularities, we can find a proper surjective morphism $g\colon T' \to T$,
with $T'$ $n$-dimensional and smooth (but possibly disconnected). Consider the Cartesian diagram
$$\xymatrix{
 X' \ar[d]_{f}\ar[r]^h & X\ar[d]^{\mu}\\
T' \ar[r]^g& T
 }$$ 
 and let $Z'=h^*(Z)$.
The assertion follows by noting that
$$W'_{\alpha}:=\big\{x\in Z'\mid\widetilde{\alpha}_x(X'_{f(x)}, Z'_{f(x)})\geq\alpha\big\}=h^{-1}(W_{\alpha}),$$
and thus $Z\smallsetminus W_{\alpha}=h(Z'\smallsetminus W'_{\alpha})$ is closed in $Z$ if $W'_{\alpha}$ is open in $Z'$. 

We now prove that $W_{\alpha}$ is open in $Z$ by induction on $\dim(T)$. The assertion is clear if $\dim(T)=0$, hence we may and will assume that $\dim(T)\geq 1$. 
We first show that the subset $W_{\alpha}\subseteq Z$ is constructible. Indeed, note first that if $T_0\subseteq T$ is a nonempty open subset that 
satisfies condition ${\rm ii_2)}$, then 
$$W_{\alpha}\cap \mu^{-1}(T_0)=\big\{x\in \mu^{-1}(T_0)\cap Z\mid \widetilde{\alpha}_x(X, Z)\geq\alpha\big\}$$
is open in $Z\cap\mu^{-1}(T_0)$. On the other hand, we can apply the induction hypothesis to the morphism $\mu^{-1}(T\smallsetminus T_0)\to T\smallsetminus T_0$
and the hypersurface $Z\cap \mu^{-1}(T\smallsetminus T_0)$
to conclude that $W_{\alpha}\cap \mu^{-1}(T\smallsetminus T_0)$ is open in $\mu^{-1}(T\smallsetminus T_0)$
(as we have discussed, the fact that $T\smallsetminus T_0$ might not be smooth is not an issue). Therefore $W_{\alpha}$ is constructible.

Since $W_{\alpha}$ is constructible, in order to prove that it is open in $Z$, it is enough to show that if $W\subseteq Z$ is an
irreducible locally closed subvariety of $Z$, of positive dimension, and $x\in W$ is such that $W\smallsetminus\{x\}\subseteq Z\smallsetminus W_{\alpha}$, then 
$x\not\in W_{\alpha}$. Of course, we may assume that $W$ dominates $T$, since otherwise we are done by induction.
Arguing by contradiction, let us assume that $x\in W_{\alpha}$.
In this case it follows from (\ref{eq1_lem_thm3_intro}) that $\widetilde{\alpha}_x(X, Z)\geq\alpha$ and thus there is an open neighborhood $V$ of $x$ in $Z$
such that $\widetilde{\alpha}_y(X, Z)\geq\alpha$ for all $y\in V$. On the other hand, if $T_0\subseteq T$ is a nonempty open subset that satisfies property ${\rm ii_2)}$,
then $W\cap\mu^{-1}(T_0)\cap V$ contains some $y\neq x$. In this case we have
$$\widetilde{\alpha}_y(X_{\mu(y)}, Z_{\mu(y)})=\widetilde{\alpha}_y(X, Z)\geq \alpha,$$
hence $y\in W_{\alpha}$, a contradiction. This completes the proof of ${\rm ii_1)}$.

The finiteness of the set 
$$\big\{\widetilde{\alpha}_x(X_{\mu(x)}, Z_{\mu(x)})\mid x\in Z\big\}$$
now follows easily by induction on $\dim(T)$, using the fact that if $T_0$ satisfies the condition in ${\rm ii_2)}$, then 
$$\big\{\widetilde{\alpha}_x(X_{\mu(x)}, Z_{\mu(x)})\mid x\in\mu^{-1}(T_0)\cap Z\big\}\subseteq \big\{\widetilde{\alpha}_x(X, Z)\mid x\in \mu^{-1}(T_0)\cap Z\big\}$$
and the right-hand side is clearly finite. Finally, if $s\colon T\to X$ is a section of $\pi$ such that $s(T)\subseteq Z$, then
$$\big\{t\in T\mid\widetilde{\alpha}_{s(t)}(X_t, Z_t)\geq\alpha\big\}=s^{-1}\left(\big\{x\in Z\mid\widetilde{\alpha}_x(X_{\mu(x)}, Z_{\mu(x)})\geq\alpha\big\}\right),$$
and thus it is open in $T$. This completes the proof of the lemma.
\end{proof}

We can now prove the properties of the minimal exponent in arbitrary codimension.

\begin{proof}[Proof of Theorem~\ref{thm3_intro}]
Since all assertions are local with respect to $X$, we may and will assume that $Z$ is defined by the ideal generated by $f_1,\ldots,f_r\in\cO_X(X)$.
Let $g=\sum_{i=1}^rf_iy_i\in \cO_X(X)[y_1,\ldots,y_r]$ and let $Z'$ be the hypersurface defined by $g$ in $U=X\times \big(\A^r\smallsetminus\{0\}\big)$.
Note that $Z\times \big(\A^r\smallsetminus\{0\}\big)\subseteq Z'$.
The plan is to use Theorem~\ref{thm2_intro} to reduce to the case of hypersurfaces.
 The only subtlety is that while the results 
concern local minimal exponents, the description provided by Theorem~\ref{thm2_intro} is not of a local nature. However, we will go around this issue using the 
homogeneity 
of $g$ in $y_1,\ldots,y_r$. More precisely, we have the following

\begin{claim}\label{claim_thm3_intro}
If $X_0\subseteq Z$ is a subset such that $\widetilde{\alpha}_{(x,\lambda)}(g)\geq \gamma$ for all $(x,\lambda)\in X_0\times \big(\A^r\smallsetminus\{0\}\big)$, 
then after replacing $X$ by an open neighborhood of $X_0$, we may assume that 
$\widetilde{\alpha}(g\vert_U)\geq\gamma$. 
\end{claim}

Indeed, consider the canonical projection
$\pi\colon U\to X\times {\mathbf P}^{r-1}$. Since $g$ is homogeneous with respect to $y_1,\ldots,y_r$, it follows that the set
$$W:=\big\{(x,\lambda)\in Z'\mid \widetilde{\alpha}_{(x,\lambda)}(g)<\gamma\}$$
is equal to $\pi^{-1}(W')$, for some subset $W'\subseteq Z\times \P^{r-1}$. 
Since $W$ is closed in $U$, we see that $W'$ is closed in $Z\times\P^{r-1}$.
By assumption, $W'\cap (X_0\times\P^{r-1})=\emptyset$. 
It follows that if $F\subseteq X$ is the projection of $W'$, then after replacing $X$ by 
$X\smallsetminus F$, which is an open neighborhood of $X_0$, we have $\widetilde{\alpha}(g\vert_U)\geq\gamma$. This proves the above claim.

Let's begin with the proof of i). Note that the hypothesis implies that $Z_H$ is a complete intersection in $H$, of pure codimension $r$, defined by the ideal generated by 
$f_1\vert_H,\ldots,f_r\vert_H$. 
Let $U_H=H\times \big(\A^r\smallsetminus\{0\}\big)$. 
If $\gamma=\widetilde{\alpha}_x(H, Z_H)$, then after replacing $X$ by a suitable open neighborhood of $x$, we may assume that
$\widetilde{\alpha}(H, Z_H)=\gamma$, hence $\widetilde{\alpha}(g\vert_{U_H})=\gamma$ by Theorem~\ref{thm2_intro}. In this case, it follows from 
\cite[Theorem~E(1)]{MP1} that $\widetilde{\alpha}_{(z,\lambda)}(g)\geq\gamma$ for every $z\in Z_H$ and every $\lambda\in\A^r\smallsetminus\{0\}$. 
We deduce using Claim~\ref{claim_thm3_intro} that after possibly replacing $X$ by a neighborhood of $Z_H$, we have
$\widetilde{\alpha}(g\vert_U)\geq\lambda$. Another application of Theorem~\ref{thm2_intro} gives $\widetilde{\alpha}(X, Z)\geq \gamma$, which completes the proof of i). 

We next prove ii). Arguing as in the proof of Lemma~\ref{lem_thm3_intro}, it is straightforward to see that if ${\rm ii_1)}$ and ${\rm ii_2)}$ hold,
then the other two assertions hold as well. Let us prove first ${\rm ii_1)}$. We need to show that for every $x\in Z$,
there is an open neighborhood $U_x$ of $x$ such that
\begin{equation}\label{eq_lower_bound}
\widetilde{\alpha}_z(X_{\mu(z)}, Z_{\mu(z)})\geq\alpha:=\widetilde{\alpha}_x(X_{\mu(x)}, Z_{\mu(x)})\quad\text{for every}\quad z\in U_x\cap Z.
\end{equation}

Let $\varphi$ be the composition $U=X\times \big(\A^r\smallsetminus\{0\}\big)\to X\overset{\mu}\longrightarrow T$. For every $t\in T$, we denote by $g_t$
the restriction of $g$ to $X_t\times \A^r$. 
After possibly replacing $X$ by an open neighborhood of $x$, we may and will assume that $\widetilde{\alpha}(X_{\mu(x)}, Z_{\mu(x)})\geq \alpha$. 
By Theorem~\ref{thm2_intro}, we have
\begin{equation}\label{eq11_thm3_intro}
\widetilde{\alpha}(g_{\mu(x)}\vert_{U_{\mu(x)}})\geq\alpha.
\end{equation}
Applying Lemma~\ref{lem_thm3_intro} for the smooth morphism
$\varphi$ and the hypersurface defined by $g$ in $U$, we see that the set
$$V_{\alpha}:=\big\{(z,\lambda)\in Z\times (\A^r\smallsetminus\{0\})\mid \widetilde{\alpha}_{(z,\lambda)}(g_{\mu(z)}\vert_{U_{\mu(z)}})\geq \alpha\big\}$$
is open in $Z\times \big(\A^r\smallsetminus\{0\}\big)$. Note that by (\ref{eq11_thm3_intro}), we have 
\begin{equation}\label{eq12_thm3_intro}
Z_{\mu(x)}\times\big(\A^r\smallsetminus\{0\}\big)\subseteq V_{\alpha}.
\end{equation}
Arguing as in the proof of Claim~\ref{claim_thm3_intro}, we see that after possibly replacing $X$ by an open neighborhood of $Z_{\mu(x)}$, we may
assume that $\widetilde{\alpha}(g_t)\geq\alpha$ for all $t\in T$ (indeed, $V_{\alpha}$ is the inverse image of an open subset  $W\subseteq Z\times\P^{r-1}$ 
and we may take the open neighborhood of $Z_{\mu(x)}$ to be the complement in $X$ of the projection of $(X\times\P^{r-1})\smallsetminus W$
onto the first component). In this case, Theorem~\ref{thm2_intro} gives $\widetilde{\alpha}(X_t, Z_t)\geq\alpha$ for all $t\in T$. 
Thus ${\rm ii_1)}$ holds. 
 
 Keeping the same notation, we now prove ${\rm ii_2)}$. Applying Lemma~\ref{lem_thm3_intro} for $\varphi$ and the hypersurface $Z'$ defined by $g$, 
 we see that there is an open subset $T_0$ of $T$ such that for every $t\in T_0$ and every $x\in Z'_t$, we have
 \begin{equation}\label{eq13_thm3_intro}
 \widetilde{\alpha}_{(x,\lambda)}(g_t)=\widetilde{\alpha}_{(x,\lambda)}(g)\quad\text{for all}\quad \lambda\in \A^r\smallsetminus \{0\}. 
 \end{equation}
 It is enough to show that for every $t\in T_0$ and every $x\in Z_t$, we have
 \begin{equation}\label{eq14_thm3_intro}
 \widetilde{\alpha}_x(X_t, Z_t)=\widetilde{\alpha}_x(X, Z).
 \end{equation}
 We fix such $t$ and $x$ and note that we may replace $X$ by any open neighborhood of $x$.
 The inequality ``$\leq$" in (\ref{eq14_thm3_intro}) always holds
  by part i) of the theorem, so we only need to prove that $\widetilde{\alpha}_x(X_t, Z_t)\geq \alpha_0:=\widetilde{\alpha}_x(X, Z)$.
  After possibly replacing $X$ by a suitable neighborhood 
 of $x$, we may and will assume that $\widetilde{\alpha}(X, Z)=\alpha_0$, hence Theorem~\ref{thm2_intro} gives $\widetilde{\alpha}(g\vert_U)=\alpha_0$. 
 By (\ref{eq13_thm3_intro}), we thus have $\widetilde{\alpha}_{(x',\lambda)}(g_{t})\geq\alpha_0$  for every $x'\in Z_t$, and 
 $\lambda\in\A^r\smallsetminus\{0\}$, and another application of Theorem~\ref{thm2_intro} gives $\widetilde{\alpha}_x(X_t, Z_t)\geq 
 \widetilde{\alpha}(X_t, Z_t)\geq\alpha_0$. This completes the proof of ${\rm ii_2)}$.

Let us prove the inequality in iii). For $\lambda=(\lambda_1,\ldots,\lambda_r)\in\A^r\smallsetminus \{0\}$, we put $g_{\lambda}=\sum_{i=1}^r\lambda_if_i$.
It follows from the assertion in ${\rm ii_2)}$ that if $\lambda$ is general, then $\widetilde{\alpha}(g_{\lambda})\geq\widetilde{\alpha}(g\vert_U)=\widetilde{\alpha}(Z)$,
where the equality follows from Theorem~\ref{thm2_intro}. Since ${\rm mult}_x(g_{\lambda})\geq k$, it follows from \cite[Theorem~E(3)]{MP1} that
$\widetilde{\alpha}_x(g_{\lambda})\leq\tfrac{n}{k}$, and thus $\widetilde{\alpha}(Z)\leq\tfrac{n}{k}$. Since the same argument applies to any open neighborhood of $x$,
we get $\widetilde{\alpha}_x(Z)\leq\tfrac{n}{k}$.
\end{proof}

\begin{rmk}
Note that the assertion in Theorem~\ref{thm3_intro}${\rm ii_1)}$ makes sense when $T$ is any (reduced, but not necessarily irreducible) variety. Moreover, the assertion in the general case can be easily reduced to the case when $T$ is smooth using resolution of singularities, as explained in the proof of Lemma~\ref{lem_thm3_intro}. 
Furthermore, 
the same argument implies that in this case, too, the set $\big\{\widetilde{\alpha}_x(X_{\mu(x)}, Z_{\mu(x)})\mid x\in Z\big\}$ is finite. 
\end{rmk}

\begin{rmk}\label{rem_sharper}
We can now see that if $X$ is a smooth, irreducible variety and $Z$ is a local complete intersection closed subscheme of $X$, of pure dimension, then for every 
singular point $x\in Z$,
we have
$$\widetilde{\alpha}_x(X,Z)\leq \dim(X)-\frac{1}{2}{\rm embdim}_x(Z),$$
where ${\rm embdim}_x(Z)=\dim_{\C}T_xZ$. Indeed, note first that if $d={\rm embdim}_x(Z)$, then after possibly replacing $X$ by an open neighborhood of 
$x$, we have a closed embedding $Z\hookrightarrow X'$, where $X'$ is smooth, irreducible, of dimension $d$. Since $x\in Z$ is a singular point, 
the ideal defining $Z$ in $X'$ is contained in $\frm_x^2$, where $\frm_x$ is the ideal defining $x$, hence Theorem~\ref{thm3_intro}iii) gives
$\widetilde{\alpha}_x(X',Z)\leq\tfrac{d}{2}$. In this case Proposition~\ref{dep_on_X} implies that 
$$\widetilde{\alpha}_x(X,Z)=\dim(X)-d+\widetilde{\alpha}_x(X',Z)\leq \dim(X)-\frac{d}{2}.$$
\end{rmk}

Our next goal is to give one nontrivial computation of minimal exponent when $r>1$. Before doing this, we give an easy lemma describing the singular locus 
of the hypersurface that we associate to a complete inersection subscheme. 
We assume that we have global coordinates $x_1,\ldots,x_n$ on the smooth variety $X$ (that is,
$dx_1,\ldots,dx_n$ trivialize $\Omega_X$) and let $\partial_{x_1},\ldots,\partial_{x_n}$ be the corresponding derivations. As usual, we suppose that we have
$f_1,\ldots,f_r\in\cO_X(X)$ that define a closed subscheme $Z$ of $X$, of pure codimension $r$, and consider $Y=X\times\A^r$, $U=X\times \big(\A^r\smallsetminus\{0\}\big)$, and
$g=\sum_{i=1}^rf_iy_i\in\cO_Y(Y)$.
 We denote by $J_f^t$ the transpose matrix of the Jacobian matrix
$\big(\partial_{x_j}(f_i)\big)_{i,j}$. 

\begin{lem}\label{singloc} 
With the above notation, the singular locus of the hypersurface $V(g)$ defined by $g$ in $Y$ is 
\[V(g)_{\rm sing}=\bigsqcup_{x\in Z}\{x\}\times W_x,\]
where $W_x={\rm Ker}\,J_f^t(x)$ is a linear subspace of $\A^r$ of dimension $\dim_{\C}(T_xZ)-\dim(Z)$, and thus
\[V(g\vert_U)_{\rm sing}=\bigsqcup_{x\in Z_{\rm sing}}\{x\}\times \big(W_x\smallsetminus\{0\}\big).\]
\end{lem}

\begin{proof} The singular locus $V(g)_{\rm sing}$ is defined by the equations 
$$\de_{y_1}(g)=\ldots=\de_{y_r}(g)=\de_{x_1}(g)=\dots=\de_{x_n}(g)=0$$
(note that these imply $g=0$ since $g$ is homogeneous of degree 1 in $y_1,\ldots,y_r$).
The formula for $V(g)_{\rm sing}$ follows from the fact that 
\[ \de_{y_i}(g) = f_i\quad\text{for}\quad 1\leq i\leq r\quad\text{and}\quad  \de_{x_j}(g) = \sum_{i=1}^r y_i \de_{x_j}(f_i)\quad\text{for}\quad 1\leq j\leq n.\]
We deduce the formula for $\dim(W_x)$  from the fact that the rank of $J_f$ at $x\in Z$ is $n-\dim_{\C}(T_xZ)$ and $J_f$ and $J_f^t$ have the same rank.
In particular, we see that $W_x\neq\{0\}$ if and only if $x\in Z_{\rm sing}$, and we obtain
 the description of $V(g\vert_U)_{\rm sing}$.
\end{proof}

\begin{eg} Let $f_1,\dots, f_r \in \C[x_1,\dots, x_n]$ be homogeneous polynomials of degree $d\geq 2$ that define a smooth, irreducible variety of codimension $r$ in ${\mathbf P}^{n-1}$. Therefore the subvariety $Z = V(f_1,\dots, f_r) \subseteq \A^n$ is a complete intersection, with a unique singular point at $0$. We will show that
\[ \widetilde{\alpha}(Z) = \frac{n}{d},\]
generalizing the well-known formula for $r=1$. 

Let $g=\sum_{i=1}^rf_iy_i$ and $U=\A^n\times \big(\A^r\smallsetminus\{0\}\big)$. 
We denote by $B_U$, respectively $\widetilde{B}_U$, the $\cD$-modules on which we have the $V$-filtration (respectively, the microlocal $V$-filtration) associated to $g\vert_U$
and we simply write $\delta_U$ for $\delta_{g\vert_U}$.
Note that it follows from Lemma~\ref{singloc} and our assumption on $Z$ that the singular locus 
of $V(g\vert_U)$ is equal to $\{0\}\times \big(\A^r\smallsetminus\{0\}\big)$. Recall that by (\ref{eq_MLCT}), if $\lambda$ is such that $\delta_U\in V^{\lambda}\widetilde{B}_U$ and
its class $\overline{\delta_U}\in{\rm Gr}_V^\lambda(\widetilde{B}_U)$ is nonzero, then $\lambda=\widetilde{\alpha}(g\vert_U)$. We will show that $\lambda=\tfrac{n}{d}$. 

Recall that for every $\alpha\in [0,1)\cap\Q$, the filtered $\cD_U$-module $\big({\rm Gr}_V^{\alpha}(B_U),F\big)$ is a filtered direct summand of a mixed Hodge module. In particular,
it is regular and quasi-unipotent along every hypersurface in the sense of \cite[Section~3.2.1]{Saito-MHP}. Moreover, its support is contained in the singular locus of $V(g\vert_U)$.
On the other hand, we have an isomorphism of filtered $\cD_U$-modules
$$\big({\rm Gr}_V^\alpha(\widetilde{B}_U),F\big) \simeq \big({\rm Gr}_V^\alpha(B_U), F\big)$$
by \cite[(2.1.4)]{Saito_microlocal}. Finally, if we write $\lambda=k+\alpha$, where $k\in\Z$ and $\alpha\in [0,1)$, it follows from \cite[(2.2.3)]{Saito_microlocal} that 
we have a filtered isomorphism 
\[ \partial_t^k: \big({\rm Gr}_V^\lambda(\widetilde{B}_U), F\big) \to \big({\rm Gr}_V^{\alpha}(\widetilde{B}_U), F[k]\big),\]
where $F[k]$ is the shifted filtration $F[k]_p=F_{p+k}$. We thus conclude that $\big({\rm Gr}^{\lambda}_V(\widetilde{B}_U),F\big)$ is regular and quasi-unipotent along every hypersurface;
moreover, its support is contained in the subset defined by $(x_1,\ldots,x_n)$. 

Note that by definition of the Hodge filtration on $\widetilde{B}_U$, we have $\delta_U\in F_0\widetilde{B}_U$ and 
$F_{-1}\widetilde{B}_U=\oplus_{i\leq -1}\cO_U\partial_t^i\delta_U$. Since $\delta_U\in V^{\lambda}\widetilde{B}_U$, it follows that
$F_{-1}\widetilde{B}_U\subseteq \sum_{i\leq -1}\partial_t^{i}\widetilde{B}_U\subseteq V^{\lambda+1}\widetilde{B}_U$. We thus see that
$F_{-1}{\rm Gr}_V^{\lambda}(\widetilde{B}_U)=0$. 

Since $\big({\rm Gr}^{\lambda}_V(\widetilde{B}_U), F\big)$ is regular and quasi-unipotent along every hypersurface, 
with support contained in the zero-locus of $x_1,\ldots,x_n$, it follows from \cite[Lemme~3.2.6]{Saito-MHP}
that $x_1,\ldots,x_n$ annihilate the first nonzero piece of the Hodge filtration 
on ${\rm Gr}^{\lambda}_V(\widetilde{B}_U)$. In particular, if $\theta_x=\sum_{i=1}^nx_i\partial_{x_i}$, so that
$\theta_x+n=\sum_{i=1}^n\partial_{x_i}x_i$, we see that $(\theta_x+n)\overline{\delta_U}=0$. 

Note that by (\ref{eq_action_B_f}), we have
$$\theta_x\delta_U=-\sum_{i=1}^nx_i\partial_{x_i}(g)\partial_t\delta_U=-dg\partial_t\delta_U,$$
where the second equality follows from the fact that $g$ is homogeneous of degree $d$ with respect to $x_1,\ldots,x_n$. 
On the other hand, using again (\ref{eq_action_B_f}), we have
$$s\delta_U=-\partial_tt\delta_U=-g\partial_t\delta_U.$$
We thus conclude that 
$$(\theta_x+n)\overline{\delta_U}=(n+sd)\overline{\delta_U}=0.$$
Since $s+\lambda$ is nilpotent on ${\rm Gr}^{\lambda}_V(\widetilde{B}_U)$, this implies $\lambda=\tfrac{n}{d}$.
\end{eg}

\section{The minimal exponent and the Hodge filtration on local cohomology}\label{secn_local_cohom}

In this section we give the proofs of Theorems~\ref{thm1_intro} and ~\ref{thm1.5_intro} by making use of a key result from~\cite{CD}, saying
that, under the assumptions of Theorem~\ref{thm1.5_intro}, we have an isomorphism 
of filtered $\cD_X$-modules
\begin{equation}\label{isom_CD}
\sigma\colon V^r(B_{\mathbf f})/(t_1,\ldots,t_r)V^{r-1}(B_{\mathbf f})\overset{\sim}\longrightarrow\cH^r_Z(\cO_X).
\end{equation}
For us, it is important to have an explicit description of this isomorphism and this is not easy to obtain
from the proof in \emph{loc. cit}. Because of this, we 
proceed in a roundabout way:
we first construct a nonzero  morphism of $\cD_X$-modules as in (\ref{isom_CD}) and then use the following 
lemma that describes the endomorphisms of $\cH^r_Z(\cO_X)$. 

\begin{lem}\label{lem_endom}
Let $X$ be a smooth, irreducible complex algebraic variety. If $Z$ is a connected, local complete intersection subscheme of $X$, of pure codimension $r$,
then the canonical map ${\mathbf C}\to {\rm End}_{\cD_X}\big(\cH^r_Z(\cO_X)\big)$ is bijective.
\end{lem}

\begin{proof}
Let $X^{\rm an}$ denote the complex manifold corresponding to $X$. In this proof we
make use of some standard results on holonomic $\cD$-modules. 
In order to prove that every endomorphism of $\cH^r_Z(\cO_X)$ is given by multiplication with a scalar, it is enough to prove
that the same property holds for ${\mathbf D}_X\big(\cH^r_Z(\cO_X)\big)$, where ${\mathbf D}_X$ is the duality functor for
holonomic $\cD_X$-modules (see \cite[Chapter~2.6]{HTT}). Moreover, by the Riemann-Hilbert correspondence 
(see \cite[Chapter~7]{HTT}), it is enough to show that every endomorphism of the perverse sheaf corresponding to 
${\mathbf D}_X\big(\cH^r_Z(\cO_X)\big)$ is given by multiplication with a scalar. For an arbitrary $Z$, this perverse sheaf is
${}^p\cH^0\big(\underline{\C}_{Z^{\rm an}}[n-r]\big)$, where $n=\dim(X)$. However, since $Z$ is locally a complete intersection, the sheaf $\underline{\C}_{Z^{\rm an}}[n-r]$
is a perverse sheaf (see \cite[Theorem~5.1.20]{Dimca}), and thus ${}^p\cH^0\big(\underline{\mathbf C}_{Z^{\rm an}}[n-r]\big)=\underline{\mathbf C}_{Z^{\rm an}}[n-r]$.
The fact that every endomorphism of $\underline{\mathbf C}_{Z^{\rm an}}$ is given by scalar multiplication follows from the fact that $Z$ is connected. 
\end{proof}

\begin{proof}[Proof of Theorem~\ref{thm1.5_intro}]
Since the assertion in the theorem can be checked locally on $X$, we may and will assume that $Z$ is connected.
The key point is to construct an explicit filtered isomorphism as in (\ref{isom_CD}). We first consider the morphism 
$$\tau\colon \cO_X[1/f_1\cdots f_r,s_1,\ldots,s_r]{\mathbf f}^{\mathbf s}\to \cO_X[1/f_1\cdots f_r]$$
given by specializing to $s_1=\ldots=s_r=-1$, that is,
$$\tau\big(P(s_1,\ldots,s_r){\mathbf f}^{\mathbf s}\big)=P(-1,\ldots,-1)\tfrac{1}{f_1\cdots f_r}.$$
It is clear that this is a morphism of left $\cD_X$-modules. 

Via the isomorphism (\ref{eq_Phi}), we will identify $\tau$ with a morphism $B_{\mathbf f}^+\to\cO_X[1/f_1\cdots f_r]$. 
Using (\ref{eq_in_terms_s}) and the fact that $Q_m(-1)=m!$, we can describe $\tau$ by
\begin{equation}\label{eq_in_terms_s2}
\tau\big(\sum_{\beta}h_{\beta}\partial_t^{\beta}\delta_{\mathbf f}\big)=\sum_{\beta}\tfrac{\beta_1!\cdots\beta_r!h_{\beta}}{f_1^{\beta_1+1}\cdots f_r^{\beta_r+1}},
\end{equation}
where $h_{\beta}\in\cO_X[1/f_1\cdots f_r]$ for all $\beta$.

It is clear from the definition that $\tau$ vanishes on $\sum_{i=1}^r(s_i+1)B_{\mathbf f}^+$. Let us consider the image of $t_i B_{\mathbf f}$. 
Note that by Lemma~\ref{lem3_appendix}, if $\beta_i\geq 1$, then 
$$t_iQ_{\beta}(s_1,\ldots,s_r){\mathbf f}^{\mathbf s}=(s_i+1)Q_{\beta'}(s_1,\ldots,s_r)f_i\cdot{\mathbf f}^{\mathbf s}\in {\rm Ker}(\tau),$$
where $\beta'=(\beta_1,\ldots,\beta_i-1,\ldots,\beta_r)$. On the other hand, if $\sum_{\beta}h_{\beta}\partial_t^{\beta}\delta_{\mathbf f}\in B_{\mathbf f}$
(so $h_{\beta}\in\cO_X$ for all $\beta$) and if $\beta$ is such that $\beta_i=0$, then 
$$\tau(t_ih_{\beta}\partial_t^{\beta}\delta_{\mathbf f})=\tau(f_ih_{\beta}\partial_t^{\beta}\delta_{\mathbf f})=h_{\beta}\prod_{j\neq i}\tfrac{\beta_j!}{f_j^{\beta_j+1}}\in
\cO_X[1/f_1\cdots \widehat{f_i}\cdots f_r].$$
We conclude that $\tau$ induces a morphism of $\cD_X$-modules
$$\overline{\tau}\colon V^rB_{\mathbf f}/\sum_{i=1}^rt_i\cdot V^{r-1}B_{\mathbf f}\to \cO_X[1/f_1\cdots f_r]/\sum_{i=1}^r\cO_X[1/f_1\cdots \widehat{f_i}\cdots f_r]
\simeq\cH^r_Z(\cO_X).$$

Let us show that $\overline{\tau}$ is surjective. Given $u=\tfrac{g}{(f_1\cdots f_r)^m}\in\cO_X[1/f_1\cdots f_r]$, for some $g\in\cO_X$ and some $m\geq 1$,
we have $u=\tau(v)$, where $v=\tfrac{g}{(m-1)!^r}\partial_t^{\beta}\delta_{\mathbf f}$, where $\beta=(m-1,\ldots,m-1)$. By the properties of the $V$-filtration, we know that 
we can find $\alpha_1,\ldots,\alpha_N\in\Q$, with $\alpha_i<r$ for all $i$, such that
$w:=(s+\alpha_1)\cdots (s+\alpha_N)v\in V^rB_{\mathbf f}$ (recall that $s=s_1+\ldots+s_r$). Since we can find $p(s)$ and $q(s)$ such that
$p(s)(s+r)+q(s)\prod_{i=1}^N(s+\alpha_i)=1$ and since $(s+r)v\in\sum_{i=1}^r(s_i+1)B_{\mathbf f}\subseteq {\rm Ker}(\tau)$, it follows that 
$u=\tau(v)=\tau\big(q(s)w\big)\in \tau(V^rB_{\mathbf f})$. Therefore we see that $\overline{\tau}$ is surjective.

As we have already mentioned, it follows from \cite[Theorems~1.1 and 1.2(b)]{CD} that we have
the isomorphism of filtered $\cD_X$-modules (\ref{isom_CD}),
where the filtration on the left-hand side is induced by the Hodge filtration on $B_{\mathbf f}$ and the filtration on the right-hand side is the
Hodge filtration that comes from the mixed Hodge module structure on $\cH^r_Z(\cO_X)$. We
recall that our convention is that the Hodge filtration on $B_{\mathbf f}$ is defined to be
\[
F_q B_{\mathbf f}= \bigoplus_{|\alpha| \leq q} \cO_X\cdot \partial^\alpha_t \delta_{\mathbf f},
\]
which differs by a shift by $r$ from the usual convention followed in~\cite{CD}.
In particular, $\overline{\tau}\circ\sigma^{-1}$ is a $\cD_X$-linear endomorphism of $\cH^r_Z(\cO_X)$; hence, by Lemma~\ref{lem_endom}, it is given by multiplication
with some $\lambda\in\C$. Since the morphism is nonzero (being surjective), it follows that $\lambda\neq 0$, hence $\overline{\tau}$ is an isomorphism. 
Moreover, since $\overline{\tau}=\lambda\sigma$, we deduce that $\overline{\tau}$ is a filtered isomorphism, too. Therefore the assertion in the theorem
follows from the definition of the Hodge filtration on $B_{\mathbf f}$ and the description of $\tau$ in (\ref{eq_in_terms_s2}).
\end{proof}

\begin{proof}[Proof of Theorem~\ref{thm1_intro}]
We shall prove the equivalent statement that $F_kB_{\mathbf f}\subseteq V^rB_{\mathbf f}$ if and only if $F_k\cH^r_Z\cO_X=O_k\cH^r_Z\cO_X$. 
The ``only if'' part is clear: since the elements
\[
\left[\frac{1}{f^{\alpha_1+1}_1f_2^{\alpha_2+1}\cdots f^{\alpha_r+1}_r}\right]\in\cH_Z^r(\cO_X)
\]
with $\alpha_1,\ldots,\alpha_r\geq 0$ and $\alpha_1+\alpha_2+\cdots+\alpha_r\leq k$ generate $O_k\cH^r_Z\cO_X$ (see for example \cite[Lemma~9.2]{MP2}),
the ``only if" part follows from Theorem~\ref{thm1.5_intro}.

For the reverse implication, we use induction on $k$. Suppose first that $k=0$. Since $\left[\tfrac{1}{f_1\cdots f_r}\right]\in F_0\cH^r_Z(\cO_X)$, it follows from 
Theorem~\ref{thm1.5_intro} that locally on $X$, we can find $h\in\cO_X$ such that $h\delta_{\mathbf f}\in V^rB_{\bff}$ and $(h-1)$ lies in the ideal ${\mathcal I}_Z$ defining $Z$. 
Therefore $F_0B_{\bff}\subseteq V^rB_{\mathbf f}$ at every point of $Z$ (and outside of $Z$, this is automatic).

Suppose now we know the assertion for $k\geq 0$ and let us prove it for $k+1$. Since $F_{k+1}\cH^r_Z(\cO_X)=O_{k+1}\cH^r_Z(\cO_X)$, it follows from 
\cite[Lemma~9.3]{MP2} that $F_k\cH^r_Z(\cO_X)=O_k\cH^r_Z(\cO_X)$, hence 
by the induction hypothesis we have
$\partial^\alpha_t\delta\in V^rB_{\mathbf f}$ for all $\alpha$ with $|\alpha|\leq k$. We need to show that
$\partial^\alpha_t\delta_{\mathbf f}\in V^rB_{\mathbf f}$ also for all $\alpha$ with $|\alpha|=k+1$. 

Since $F_{k+1}\cH^r_Z(\cO_X)=O_{k+1}\cH^r_Z(\cO_X)$, it follows from Theorem~\ref{thm1.5_intro} that the map
$$\sigma\colon F_{k+1}V^rB_{\bff}\to {\rm Gr}^O_{k+1}\big(\cH^r_Z(\cO_Z)\big), \quad \sum_{|\beta|\leq k+1}h_{\beta}\partial_t^{\beta}\delta_{\bff}\to
\sum_{|\beta|=k+1}\left[\tfrac{\beta_1!\cdots \beta_r!h_{\beta}}{f_1^{\beta_1+1}\cdots f_r^{\beta_r+1}}\right]$$
is surjective. This implies that working locally on $X$, for every $\alpha$ with $|\alpha|=k+1$ we can find 
$$u_{\alpha}=\sum_{|\beta|\leq k+1}h_{\alpha,\beta}\partial_t^{\beta}\delta_{\bff}\in V^rB_{\bff}$$
that is mapped by $\sigma$ to $v_{\alpha}=\left[\tfrac{\alpha_1!\alpha_2!\cdots \alpha_r!}{f_1^{\alpha_1+1}\cdots f_r^{\alpha_r+1}}\right]\in 
{\rm Gr}^O_{k+1}\big(\cH^r_Z(\cO_Z)\big)$. 
Note that since $Z$ is a complete intersection, ${\rm Gr}^O_{k+1}\cH^r_Z(\cO_Z)$ is a free $\cO_Z$-module, with a basis given by the $v_{\alpha}$, with $|\alpha|=k+1$
(see for example \cite[Lemmas~9.1, 9.2]{MP2}). 
This implies that $h_{\alpha,\alpha}-1$ and $h_{\alpha,\beta}$, for $\alpha\neq\beta$, lie in $\cI_Z$. It is then clear that for every $\alpha$ with $|\alpha|=k+1$ we have
$\partial_t^{\alpha}\delta_{\bff}\in V^rB_{\bff}$ at every point of $Z$ (this holds trivially on the complement of $Z$). This completes the proof of the induction step and thus the proof of the 
``if" part.
\end{proof}

We obtain the following consequence to the characterization of local complete intersection Du Bois singularities. We note that the ``if" part is
\cite[Corollary~5.8]{Schwede} and the full equivalence in the case when $Z$ is normal is \cite[Theorem~3.6]{Kovacs} (note that 
if $Z$ is normal and local complete intersection, then
$\lct(X,Z)=r$
if and only if $Z$ has log canonical singularities by Inversion of Adjunction, see \cite[Corollary~3.2]{EM}). A version of the ``only if" implication
also appears in \cite[Theorem~4.2]{Doherty}. 

\begin{cor}\label{cor_IA}
If $X$ is a smooth, irreducible variety and $Z\hookrightarrow X$ is a local complete intersection closed subscheme of pure codimension $r$, then 
$Z$ has Du Bois singularities\footnote{We note that the condition of having Du Bois singularities assumes, in particular, that $Z$ is reduced.} if and only if
$\lct(X,Z)=r$.
\end{cor}

\begin{proof}
Recall first that $\lct(X,Z)\leq r$ and equality implies that $Z$ is reduced: see Remarks~\ref{lct_bounded_codim} and \ref{rel_with_lct2}.
We may thus assume that $Z$ is reduced.
Since $Z$ is locally a complete intersection, it follows from \cite[Theorem~C]{MP2} that $Z$ has Du Bois singularities if and only if
$F_0\cH^r_Z(\cO_X)=O_0\cH^r_Z(\cO_X)$ and this condition is equivalent to $\lct(X,Z)\geq r$ by Theorem~\ref{thm1_intro}.
\end{proof}

\section{The minimal exponent and the Bernstein-Sato polynomial}

Let $X$ be a smooth, irreducible, complex algebraic variety and $Z$ a proper (nonempty) closed subscheme of $X$, defined by the ideal $\fra$.
In what follows we will make use of the notation and definitions introduced in Section~\ref{section_review}. Recall, in particular, that
we discussed
 the Bernstein-Sato polynomial $b_Z(s)$ in the case when $\fra$ is generated by nonzero global regular functions $f_1,\ldots,f_d$. 
The general case can be easily reduced to this one (in fact, to the case when $X$ is affine) since for open subsets $U_1,\ldots,U_N$ of $X$
such that $Z\subseteq U_1\cup\ldots\cup U_N$, we have
\begin{equation}\label{eq_bfcn_union}
b_Z(s)={\rm LCM}\big\{b_{Z\cap U_i}(s)\mid 1\leq i\leq N\big\}
\end{equation}
(with the convention that $b_{Z\cap U_i}(s)=1$ if $Z\cap U_i=\emptyset$).

\begin{prop}\label{prop_r}
If $Z$ is locally a complete intersection in $X$, of pure codimension $r$, then $b_Z(-r)=0$.
\end{prop}

\begin{proof}
If $U$ is an open subset of $X$, then we deduce from (\ref{eq_bfcn_union}) that $b_{Z\cap U}(s)$ divides $b_Z(s)$. It follows that
after replacing $X$ by a suitable affine open neighborhood of a point in $Z$, we may assume that $X$ is affine and the ideal $\fra$ defining $Z$ is generated by
a regular sequence $f_1,\ldots,f_r\in\cO_X(X)$. The condition in the definition of $b_Z(s)$ can be reformulated as saying that $b_Z(s)$ is the monic polynomial
of minimal degree such that
\begin{equation}\label{eq1_prop_r}
b_Z(s)f_1^{s_1}\cdots f_r^{s_r}\in\sum_{u\in\Z^r,|u|=1}\cD_X[s_1,\ldots,s_r]\cdot\prod_{u_i<0}{{s_i}\choose {-u_i}}f_1^{s_1+u_1}\cdots f_r^{s_r+u_r}
\end{equation}
(see \cite[Section 2.10]{BMS}). Here we use the isomorphism (\ref{eq_Phi}),
that maps $\delta_{\bff}$ to $f_1^{s_1}\cdots f_r^{s_r}$, such that the action of $-\partial_{t_i}t_i$ corresponds to the action of $s_i$ (so that $s=s_1+\ldots+s_r$). 
Note also that ${{s_i}\choose {-u_i}}$ denotes the polynomial $\tfrac{1}{(-u_i)!}s_i(s_i-1)\cdots (s_i+u_i-1)$.

By making $s_1=\ldots=s_r=-1$ in (\ref{eq1_prop_r}), we obtain the following relation in $\cO_X[1/f_1\cdots f_r]$:
\begin{equation}\label{eq2_prop_r}
b_Z(-r)\frac{1}{f_1\cdots f_r}\in\sum_{u\in\Z^r, |u|=1}\cD_X\cdot f_1^{u_1-1}\cdots f_r^{u_r-1}.
\end{equation}
Note that if $u=(u_1,\ldots,u_r)$ is such that $|u|=1$, then there is $i$ such that $u_i\geq 1$, in which case 
$$\cD_X\cdot f_1^{u_1-1}\cdots f_r^{u_r-1}\subseteq \cD_X\cdot \prod_{j\neq i}f_j^{u_j-1}\subseteq \cO_X[1/f_1\cdots\widehat{f_i}\cdots f_r].$$
Using (\ref{eq2_prop_r}), we thus conclude that
$$
b_Z(-r)\frac{1}{f_1\cdots f_r}\in\sum_{i=1}^r\cO_X[1/f_1\cdots\widehat{f_i}\cdots f_r],
$$
hence the class of $b_Z(-r)\frac{1}{f_1\cdots f_r}$ in the local cohomology $\cH^r_Z(\cO_X)$ is $0$. Since the class of $\frac{1}{f_1\cdots f_r}$
in the local cohomology is nonzero, we conclude that $b_Z(-r)=0$.
\end{proof}

From now on, for the rest of this section, we assume that $Z$ is a local complete intersection closed subscheme of $X$, of pure codimension $r\geq 1$.
The above result motivates the following 

\begin{defi}
We denote by $\widetilde{\gamma}(Z)$ the negative of the largest root of $b_Z(s)/(s+r)$
(with the convention that $\widetilde{\gamma}(Z)=\infty$ if this polynomial is $1$).
\end{defi}

\begin{rmk}\label{rel_with_lct}
Note that since all roots of $b_Z(s)$ are rational numbers, the same is true for $\widetilde{\gamma}(Z)$.
Recall also that by formula (\ref{eq_lct_b_function}), the largest root of $b_Z(s)$ is $-\lct(X,Z)$, hence
\begin{equation}\label{eq_rel_with_lct}
\lct(X,Z)=\min\big\{\widetilde{\gamma}(Z),r\}.
\end{equation}
\end{rmk}

\begin{rmk}\label{rel_with_rat_sing}
Since $Z$ is locally a complete intersection in $X$, of pure codimension $r$, it follows from \cite[Theorem~4]{BMS} that $\widetilde{\gamma}(Z)>r$
if and only if $Z$ has rational singularities (the result in \emph{loc}. \emph{cit}. requires $Z$ to also be reduced, but the hypothesis $\widetilde{\gamma}(Z)>r$
implies $\lct(X,Z)=r$, and thus $Z$ is reduced by Remark~\ref{rel_with_lct2}).
\end{rmk}

We will need the following easy result: 

\begin{lem}\label{lem_reg_seq}
If $X={\rm Spec}(R)$ is a smooth affine variety and $f_1,\ldots,f_r\in R$ form a regular sequence, then for every $k\geq 0$, the sequence $t_1,\ldots,t_r$
is regular on the finitely generated $R[t_1,\ldots,t_r]$-module $F_kB_{\bff}$.
\end{lem}

\begin{proof}
For every $k\geq 0$, the quotient 
$F_kB_{\bff}/F_{k-1}B_{\bff}$ is a free $R$-module. 
Moreover, each $t_i$ acts on this quotient as multiplication by $f_i$. The assertion in the lemma thus follows
by induction on $k$, using the fact that if 
$$0\to M'\to M\to M''\to 0$$
is a short exact sequence of $R[t_1,\ldots,t_r]$-modules such that $t_1,\ldots,t_r$ is a regular sequence on both $M'$ and $M''$,
then it is a regular sequence also on $M$.
\end{proof}

We can now prove our result relating $\widetilde{\gamma}(Z)$ and $\widetilde{\alpha}(Z)$.

\begin{proof}[Proof of Theorem~\ref{thm4_intro}]
After possibly replacing $X$ by suitable affine open subsets, we may and will assume that $X$ is affine
and the ideal defining $Z$ is generated by a regular sequence $f_1,\ldots,f_r$. 
It follows from (\ref{eq_rmk_min_exp_lct}) and (\ref{eq_rel_with_lct}) that we have
$$\min\big\{\widetilde{\gamma}(Z),r\}=\lct(X,Z)=\min\big\{\widetilde{\alpha}(Z),r\}.$$
Therefore both assertions in the theorem hold if $\lct(X,Z)<r$. Hence from now on we may
and will assume that $\lct(X,Z)=r$, which in light of (\ref{formula_lct}), is equivalent to $\delta_{\bff}\in V^rB_{\bff}$. In order to prove the two assertions in the theorem, 
it is enough to show the following:
\begin{enumerate}
\item[(i)] If $q$ is a nonnegative integer and $\gamma\in (0,1]$ is a rational number such that $\widetilde{\gamma}(Z)\geq r-1+q+\gamma$,
then $\widetilde{\alpha}(Z)\geq r-1+q+\gamma$; by definition of the minimal exponent, this is equivalent to $\partial_t^{\beta}\delta_{\bff}\in V^{r-1+\gamma}B_{\bff}$
for all $\beta\in\Z_{\geq 0}^r$ with $|\beta|\leq q$.
\item[(ii)] If $\gamma'\in (0,1]$ is a rational number such that $\widetilde{\alpha}(Z)\geq r+\gamma'$, then $\widetilde{\gamma}(Z)\geq r+\gamma'$. 
\end{enumerate}

We first prove (i), arguing by induction on $q\geq 0$. If $q=0$, then we are done, since we are assuming $\delta_{\bff}\in V^rB_{\bff}$. Suppose now that $q\geq 1$.
By the induction hypothesis, it is enough 
to show that for every $\beta\in\Z_{\geq 0}^r$, with $|\beta|=q-1$, if $u=\partial_t^{\beta}\delta_{\bff}$, then $\partial_{t_i}u\in V^{r-1+\gamma}B_{\bff}$
for $1\leq i\leq r$. Furthermore, the induction hypothesis gives $u\in V^rB_{\bff}$. Let's write $b(s):=b_Z(s)=(s+r)p(s)$.

By definition of $b_Z(s)$, we have
\begin{equation}\label{eq1_thm4_intro}
b(s)V^0\cR\cdot \delta_{\bff}\subseteq V^1\cR\cdot\delta_{\bff}.
\end{equation}
We first note that for every $i$ with $1\leq i\leq q-1$, if $\beta'\in\Z_{\geq 0}^r$ is such that $|\beta'|=q-i$, then there is $\beta''\in\Z^r_{\geq 0}$ with 
$|\beta''|=q-i-1$ such that 
\begin{equation}\label{eq2_thm4_intro}
b(s-q+i)\partial_t^{\beta'}V^0\cR\cdot\delta_{\bff}\subseteq \partial_t^{\beta''}V^0\cR\cdot \delta_{\bff}.
\end{equation}
Indeed, it follows from Lemma~\ref{lem2_appendix} that if we take any $\beta''\in\Z_{\geq 0}^r$ with $|\beta''|=q-i-1$
and $\beta'_j\geq\beta''_j\geq 0$ for all $j$, then
$$b(s-q+i)\partial_t^{\beta'}V^0\cR\cdot\delta_{\bff}=\partial_t^{\beta'}b(s)V^0\cR\cdot\delta_{\bff}
\subseteq \partial_t^{\beta'}V^1\cR\cdot\delta_{\bff}\subseteq\partial_t^{\beta''}V^0\cR\cdot\delta_{\bff}.$$
Applying (\ref{eq2_thm4_intro}) for all $i$ with $1\leq i\leq q-1$, as well as (\ref{eq1_thm4_intro}), we obtain
$$b(s)b(s-1)\cdots b(s-q+1)\partial_t^{\beta}\delta_{\bff} \subseteq b(s)V^0\cR\cdot \delta_{\bff}\subseteq V^1\cR\cdot\delta_{\bff}\subseteq V^{r+1}B_{\bff}.$$
Note now that we can write 
$$b(s)b(s-1)\cdots b(s-q+1)=(s+r)p_1(s),$$
where 
$$p_1(s)=(s+r-1)\cdots (s+r-q+1)\cdot\prod_{i=0}^{q-1}p(s-i).$$
The assumption that $\widetilde{\gamma}(Z)\geq r-1+q+\gamma$ implies that $p_1(s)$ has no roots in the interval $(-r-\gamma,-r]$. Since $(s+r)u\in V^rB_{\bff}$ (recall that $u\in V^rB_{\bff}$) and $p_1(s)(s+r)u\in V^{r+1}B_{\bff}$, we conclude that
$(s+r)u\in V^{r+\gamma}B_{\bff}$. By assumption, we have $u\in F_{q-1}B_{\bff}$, hence $(s+r)u\in F_qV^{r+\gamma}B_{\bff}$ (where for every $p\in\Z_{\geq 0}$ and every
$\alpha\in\Q$, we put $F_pV^{\alpha}B_{\bff}=F_pB_{\bff}\cap V^{\alpha}B_{\bff}$). 

Consider now the map
$$(F_qV^{r-1+\gamma}B_{\bff})^{\oplus r}\overset{(t_1,\ldots,t_r)}\longrightarrow F_qV^{r+\gamma}B_{\bff}.$$
Since $\gamma>0$, it follows from \cite[Theorem~1.1]{CD} that the map is surjective, hence we may write
$$(s+r)u=\sum_{i=1}^rt_iu_i,$$
with $u_i\in F_qV^{r-1+\gamma}B_{\bff}$ for $1\leq i\leq r$. 
Since $s+r=-\sum_{i=1}^rt_i\partial_{t_i}$, we obtain
$$\sum_{i=1}^rt_i(u_i+\partial_{t_i}u)=0.$$
Using the fact that $t_1,\ldots,t_r$ form a regular sequence on $F_qB_{\bff}$ by Lemma~\ref{lem_reg_seq}, we conclude that
for every $i$, we have
$$u_i+\partial_{t_i}u\in \sum_{i=1}^rt_i\cdot F_qB_{\bff}\subseteq V^rB_{\bff},$$
where the inclusion follows from the fact that we already know, by induction, that $F_{q-1}B_{\bff}\subseteq V^rB_{\bff}$ and thus $F_qB_{\bff}\subseteq 
\sum_{i=1}^r\partial_{t_i}\cdot V^rB_{\bff}\subseteq V^{r-1}B_{\bff}$. 
We thus conclude that
$$\partial_{t_i}u=(u_i+\partial_{t_i}u)-u_i\in V^{r-1+\gamma}B_{\bff}.$$
This completes the proof of (i). 

We next prove (ii). We will make use of the results in Section~\ref{section_rel_Vfilt}.
By definition of $\widetilde{\alpha}(Z)$, we know that $\partial_{t_i}\delta_{\bff}\in V^{r-1+\gamma'}B_{\bff}$ for $1\leq i\leq r$.
Theorem~\ref{thm_V_filtration} (see also equation (\ref{eq_comp_V_fil})) 
thus gives $y_i\delta_g\in V^{r+\gamma'}\widetilde{B}_g$ for $1\leq i\leq r$. By Lemma~\ref{lem_theta}, we have
$v:=(s+r)\delta_g=\sum_{i=1}^r\partial_{y_i}y_i\delta_g \in V^{r+\gamma'}\widetilde{B}_g$. Using the description (\ref{eq2_formula_V_filtration}), we deduce that all roots of 
$\widetilde{b}_v(s)$ are $\leq -(r+\gamma')$. On the other hand, the inclusions
$$\widetilde{b}_v(s)(s+r)\delta_g\subseteq V^1\widetilde{\cR}\cdot (s+r)\delta_g\subseteq V^1\widetilde{\cR}\cdot\delta_g,$$
imply that $\widetilde{b}_{\delta_g}$ divides $\widetilde{b}_v(s)(s+r)$. 
Since $b_Z(s)=\widetilde{b}_{\delta_g}$ (see Remark~\ref{rmk_relation_b_function}), we conclude that
$\widetilde{\gamma}(Z)\geq r+\gamma'$. This completes the proof of the theorem.
\end{proof}

We can now show that the equality $F_1=O_1$ on $\cH^r_Z(\cO_X)$ implies that $Z$ has rational singularities.

\begin{proof}[Proof of Corollary~\ref{cor_intro}]
It follows from Theorem~\ref{thm4_intro} that $\widetilde{\alpha}(Z)>r$ if and only if $\widetilde{\gamma}(Z)>r$. By \cite[Theorem~4]{BMS}, this holds if and only if 
$Z$ has rational singularities (see also Remark~\ref{rel_with_rat_sing}). The last assertion in the corollary now follows from Theorem~\ref{thm1_intro}.
\end{proof}

\section{Appendix: some formulas involving differential operators}

In this appendix we collect for ease of reference some easy computations involving differential operators. 
We work in the ring $\C\langle t,\partial_t,\partial_t^{-1}\rangle$ and put
 $s=-\partial_tt$. 

\begin{lem}\label{lem0_appendix}
For every $m\geq 1$ we have
$[\partial_t,t^m]=mt^{m-1}$ and for every $m\in\Z$ we have $[t,\partial_t^m]=-m\partial_t^{m-1}$.
\end{lem}

\begin{proof}
The first formula follows from the more general fact that for every derivation $D$ and every regular function $f$, we have $[D,f]=D(f)$.
The second formula is clear if $m=1$ and the general case follows by induction on $|m|$, using the fact that
$$[t,\partial_t^{m+1}]=(t\partial_t^m-\partial_t^mt)\partial_t+\partial_t^m(t\partial_t-\partial_tt)=[t,\partial_t^m]\partial_t-\partial_t^m,$$
which immediately implies that the formula holds for $m$ if and only if it holds for $m+1$.
\end{proof}

\begin{lem}\label{lem1_appendix}
For every $m\geq 1$, we have 
$$\partial_t^mt^m=(-1)^ms(s-1)\cdots (s-m+1).$$
\end{lem}

\begin{proof}
The assertion is clear when $m=1$ and the general case follows by induction on $m$, writing
$$\partial_t^{m+1}t^{m+1}=\partial_tt\partial_t^mt^m+\partial_t[\partial_t^m,t]t^m=(\partial_tt+m)\partial_t^mt^m,$$
where the last equality follows using Lemma~\ref{lem0_appendix}.
\end{proof}

\begin{lem}\label{lem2_appendix}
For every $m\in\Z$ and every $P\in\C[s]$, we have $P(s)\partial_t^m=\partial_t^mP(s+m)$.
\end{lem}

\begin{proof}
It is easy to see that it is enough to prove the equality when $P$ is a monomial and then that it is enough to prove it when $P=s$. 
In this case we have
$$s\partial_t^m-\partial_t^ms=-\partial_tt\partial_t^m+\partial_t^{m+1}t=-\partial_t[t,\partial_t^m]=m\partial_t^m,$$
where the last equality follows from Lemma~\ref{lem0_appendix}.
\end{proof}

The proof of the next lemma is similar and we leave it for the reader.

\begin{lem}\label{lem3_appendix}
For every $m\in\Z_{\geq 0}$ and every $P\in\C[s]$, we have $P(s)t^m=t^mP(s-m)$.
\end{lem}

\end{document}